\documentclass[11pt]{amsart}

\usepackage[marginratio=1:1,height=600pt,width=374pt,tmargin=110pt]{geometry}

\usepackage{amsthm}
\usepackage{breqn}
\usepackage{amsmath}
\usepackage{amsfonts}
\usepackage{amssymb}
\usepackage{tikz}
\usepackage{hyperref}
\usepackage{microtype} %Avoids broken equations in text
\usepackage{mathtools}
\usepackage{mathrsfs}
\usepackage{graphicx}
\usepackage{stmaryrd}
\usepackage{cleveref}

\newcommand{\G}{\mathscr{G}}
\newcommand{\R}{\mathbb{R}} 
\newcommand{\e}{\varepsilon}
\newcommand{\N}{\mathbb N}
\newcommand{\Z}{\mathbb Z}

\newcommand{\M}{\widetilde{M}}

\swapnumbers

\newtheorem{theorem}[subsection]{Theorem}
\newtheorem{lemma}[subsection]{Lemma}
\newtheorem{corollary}[subsection]{Corollary}
\newtheorem{proposition}[subsection]{Proposition}

\newtheorem{definition}[subsection]{Definition}
\newtheorem{remark}[subsection]{Remark}

\newtheorem*{claim*}{Claim}
\newtheorem*{theorem*}{Theorem}
\newtheorem*{remark*}{Remark}

\title[\resizebox{5in}{!}{Mean curvature flow in Homology and foliations of hyperbolic $3$-manifolds \ \ }]{Mean curvature flow in Homology and foliations of hyperbolic $3$-manifolds}
\author{Marco A. M. Guaraco, Vanderson Lima, Franco Vargas Pallete}
\address{Imperial College London}
\email{guaraco@imperial.ac.uk}

\address{Instituto de Matem\'atica e Estat\'istica\\ 
Universidade Federal do Rio Grande do Sul\\ 
Brazil}
\email{vanderson.lima@ufrgs.br}

\address{Yale University}
\email{franco.vargaspallete@yale.edu}

\thanks{FVP was supported by the Minerva Research Foundation and by NSF grant DMS-2001997.}

\begin{document}
 
%%%%%%%%%%%%%%%%%%%%ABSTRACT

\setcounter{tocdepth}{1}
\begin{abstract}
We study global aspects of the mean curvature flow of non-separating hypersurfaces $S$ in closed manifolds. For instance, if $S$ has non-vanishing mean curvature, we show its level set flow converges smoothly towards an embedded minimal hypersurface $\Gamma$. We prove a similar result for the flow with surgery in dimension 2. As an application we show the existence of monotone incompressible isotopies in manifolds with negative curvature. Combining this result with min-max theory, we show that quasi-Fuchsian and hyperbolic $3$-manifolds fibered over $\mathrm{S}^1$ admit smooth entire foliations whose leaves are either minimal or have non-vanishing mean curvature. We also conclude the existence of outermost minimal surfaces for quasi-Fuchsian ends and study their continuity with respect to variations of the quasi-Fuchsian metric. 

\end{abstract}

\maketitle 

\vspace{-28pt}

\section{Main results}

Let $S$ be an embedded hypersurface in a Riemannian 
manifold $M$. It is an interesting question whether one can evolve $S$ by the mean curvature flow towards an embedded minimal hypersurface $\Gamma$. As singularities naturally emerge during the evolution, weak notions of mean curvature flow are necessary in order to continue the process for large times. In some cases, regularity theories for weak flows are developed under the crucial assumption that the evolving hypersurface is a boundary. A natural question in this context is whether those results can be extended to non-separating evolutions.  

We first look at the \textit{level set flow} approach (also called ``biggest flow'' in \cite{IlmanenShort}). In \cite{WhiteSize}, B. White developed a regularity theory for this flow assuming $S$ is the boundary of a compact mean-convex set. One of the key observations in this case is that the evolution is monotone and encloses a region satisfying an outward minimizing property. When $S$ is non-separating the global geometry of $M$ might complicate this description: not only $S$ does not bound a region but its evolution might repeat points of $M$ (see \Cref{figrepeat}).

Throughout this paper we assume that $3\leq \dim M\leq 7$ and say that  $S\subset M$ has \textit{non-vanishing mean curvature} if its mean curvature vector satisfies $\vec{H}(p) \neq 0, \forall\, p \in S$. 

Our first result is:

\begin{theorem}\label{theoremHomology}
Let $S\subset M$ be a non-separating hypersurface with non-vanishing mean curvature. If $t\mapsto F_t(S)$ is the level set flow of $S$, then there exists a stable minimal hypersurface $\Gamma$, such that $F_t(S) \to \Gamma$ smoothly as $t\to\infty$. 
\end{theorem}

Next, we study in more detail the evolution of a 2-dimensional surface $S$. For this, we look at \textit{the mean-convex mean curvature flow with surgery} from \cite{HaslhoferKetover,HaslhoferKleiner2} (a similar theory was developed by \cite{BrendleHuisken3m}). First, we extend the results from \cite{HaslhoferKetover,BrendleHuisken3m} to the non-separating case:

\begin{theorem}\label{theoremSurgery}
Let $M$ be a closed $3$-manifold and $S\subset M$ a non-separating surface with non-vanishing mean curvature. Let $\Gamma$ be the limit of the level set flow of $S$. Then, there exists a mean curvature flow with surgery starting at $S$ such that, after finitely many surgeries, converges smoothly to $\Gamma$ with the same multiplicity as the level set flow.
\end{theorem}

The mean-convex mean curvature flow with surgery has been used in \cite{HaslhoferKetover} to construct smooth mean-convex foliations of $3$-balls by $2$-spheres (see also \cite{BuzanoSpheres,BuzanoTori,LiokumovichMaximo} for other applications). Building on these ideas, our next goal is to prove the existence of strictly monotone mean-convex isotopies of incompressible surfaces, both in the separating and non-separating case (see \Cref{defisotopy}). More precisely, we show:

\begin{theorem}\label{theoremIncompressible}
Let $M$ be a closed $3$-manifold and $S\subset M$ a connected incompressible surface with non-vanishing mean curvature. Let $\Gamma$ be the limit of the level set flow of $S$. 

Then, $\Gamma$ contains an incompressible component $\Gamma'$ such that: 
\begin{enumerate}
\item [(a1)] $\Gamma'$ has multiplicity $1$ and is diffeomorphic to $S$, or
\item [(a2)] $\Gamma'$ has multiplicity $2$, is one-sided and the boundary of a tubular neighborhood of $\Gamma'$ is diffeomorphic to $S$. 
\item [(b)] $\Gamma\setminus \Gamma'$ is a (possibly empty) union of copies of $\mathrm{S}^2$ and $\mathbb{RP}^2$.
\end{enumerate}

Moreover, if $(a1)$ holds and $\Gamma\setminus \Gamma'=\emptyset$, then there exists a strictly monotone mean-convex isotopy between $S$ and $\Gamma$.

\end{theorem}

This result suits particularly well the context of negatively curved $3$-manifolds. In this case, minimal surfaces cannot be diffeomorphic to $\mathrm{S}^2$ or $\mathbb{RP}^2$ as a consequence of the Gauss formula and the Gauss-Bonnet Theorem. Therefore, $\Gamma\setminus \Gamma'=\emptyset$ holds automatically. We show:

\begin{theorem}\label{theoremNegative}
Let $M$ be a $3$-manifold with negative sectional curvature and $S\subset M$ an incompressible surface with non-vanishing mean curvature. Assume that one of the following holds:
\begin{enumerate}
\item $M$ is closed and $S$ is non-separating, or
\item $K\subset M$ is a mean-convex component of $\partial K$, $K\subset M$ is compact and the components of $ \partial K\setminus S \neq \emptyset$ are either minimal or mean-convex.
\end{enumerate} Then, there exists a strictly monotone mean-convex isotopy between $S$ and the stable minimal surface $\Gamma$.
\end{theorem}

Combining this result with min-max theory we obtain:

\begin{theorem}\label{theoremFoliation}
Let $M$ be a hyperbolic $3$-manifold and $S$ a surface of genus at least $2$. Assume that:
\begin{enumerate}
\item $M\simeq S\times \R$, with a quasi-Fuchsian metric or
\item $M$ fibers over $\mathrm{S}^1$, with fiber $S$.
\end{enumerate} 
Then, $M$ admits a smooth foliation by surfaces on the isotopy class of the fiber, such that each leaf is either minimal or has non-vanishing mean curvature. 
\end{theorem}

The construction of these foliations was one of the original motivations we had in mind when we began this work. The ends of a general quasi-Fuchsian manifold admit a unique foliation by constant mean curvature surfaces isotopic to the fiber \cite{MazzeoPacard}. One might wonder if this constant mean curvature foliation can be extended globally. This is known to be impossible in certain cases \cite{HuangWang4}. A weaker question is whether one can construct global foliations such that the mean curvature of each leaf does not change sign. \Cref{theoremFoliation} answers this  affirmatively.

As a byproduct, we obtain the existence of outermost minimal surfaces in quasi-Fuchsian ends. We finish this work by studying the continuity properties of these surfaces with respect to perturbations in the family of quasi-Fuchsian metrics. We summarize our findings as:

\begin{theorem}\label{theoremOutermost}
A marked end of a quasi-Fuchsian manifold $(S\times \R,h)$ admits a minimal surface $\Sigma_{out}(h)$ which is outermost with respect to that end. The map $h\mapsto \Sigma_{out}(h)$ is continuous on an open set of total measure of the family of quasi-Fuchsian metrics. Discontinuities of the map $h\mapsto \Sigma_{out}(h)$ exist and in those cases $\Sigma_{out}(h)$ is degenerate stable and not locally area-minimizing. Finally, the function $h\mapsto \operatorname{Area}(\Sigma_{out}(h))$ is upper semi-continuous.
\end{theorem}

Although we explicitly consider the case of quasi-Fuchsian manifolds of the form $S\times \R$, where $S$ is a closed surface of genus $g\geq 2$, our existence and continuity results concerning the outermost minimal surface are valid for more general hyperbolic 3-manifolds. For instance, these results apply to expanding ends with incompressible boundary in convex cocompact hyperbolic 3-manifolds (see \cite{CanaryEnds}).

\begin{remark}
\Cref{theoremOutermost} shows complete hyperbolic metrics can contain degenerate stable minimal surfaces that do not look like strictly stable or unstable minimal surfaces, but rather as ``saddle'' type surface. In the notiation of \Cref{defcontracting}, these surfaces have a tubular neighborhood of ``mixed type'' (see also surface IV in \Cref{figQF}).
\end{remark}

\noindent
{\bf Acknowledgments}: We are grateful to the Institute for Advanced Study where this work was initiated during the Special Year of Variational Methods in Geometry.  We thank R. Haslhofer for his kindness and patience during consultations and for suggesting the problem of the mean curvature flow with homology. We thank F. Schulze for helpful discussions and I. Agol, D. Calegari, B. Farb, F. C. Marques, L. Mazet, A. Neves, H. Rosenberg and K. Uhlenbeck for their interest in this work.

\section{Examples, ideas and organization}

The flow of a non-separating surface might repeat points of $M$, even when the surface has non-vanishing mean curvature. This is illustrated in \Cref{figrepeat}. 
\begin{figure}[h]
\centering
  \includegraphics[width=130pt]{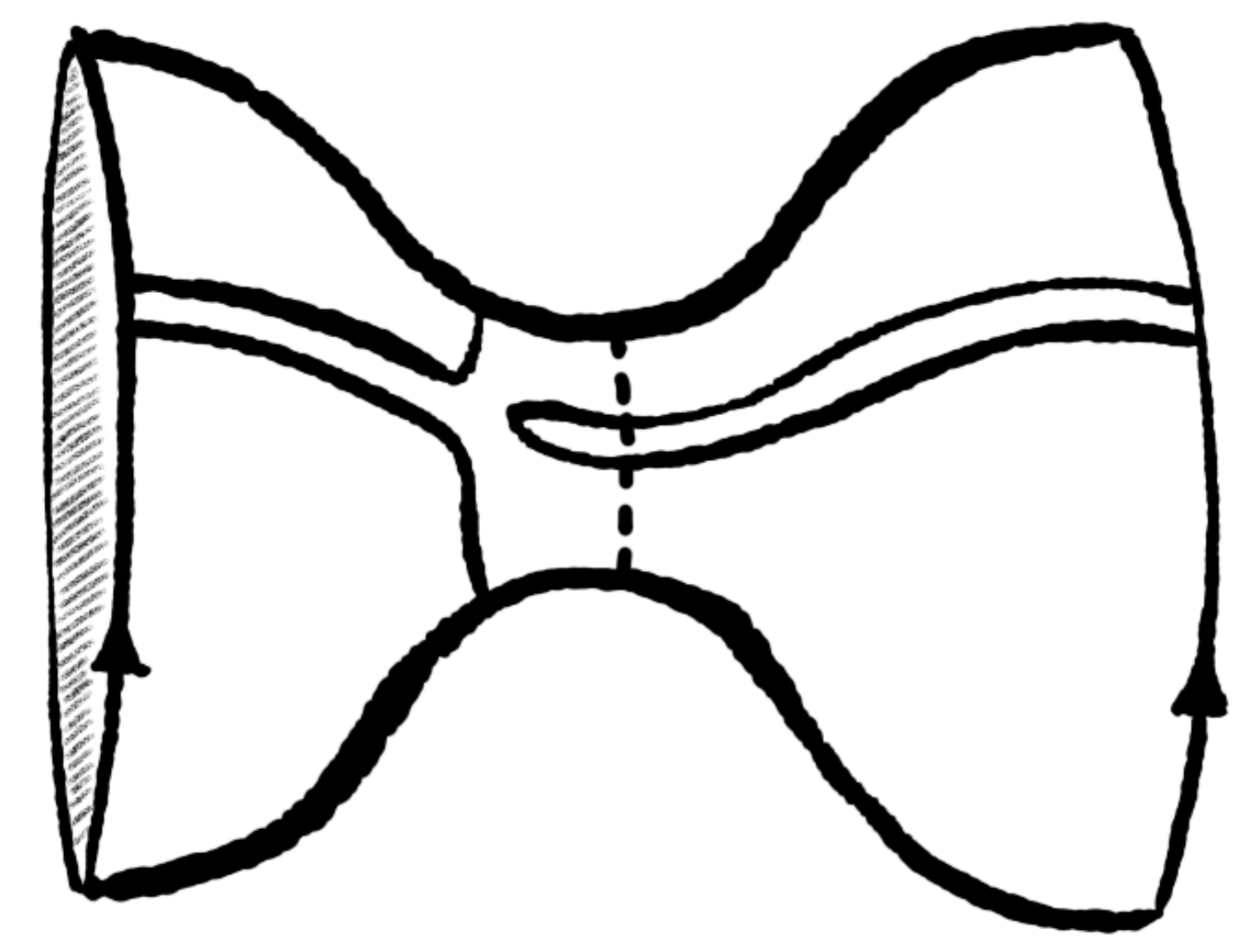}
  \caption{A hypersurface with a narrow tube (continuous line) that intersects an area-minimizing hypersurface (dashed line). The tube disappears quickly along the flow and the hypersurface eventually converges to the area-minimizing. This evolution repeats points.}
  \label{figrepeat}
 \end{figure}
 To overcome this problem, we consider a covering space $\pi: \M\to M$ with the property that the class $[S]\subset H_{n-1}(M,\Z)$ lifts to a separating class in $\M$. This covering is illustrated in \Cref{figCovering}. Its main properties can be found in \Cref{coveringdef}, \Cref{lifts}, \Cref{remarklift}. The main observation is that the lift of area-minimizing hypersurfaces in $[S]$ produce minimal barriers restricting level set flows to compact regions of $\M$. This is discussed in \Cref{minimalbarriers}. Regularity theories for compact evolutions can be applied to the region bounded by one of the lift of $S$, which is an unbounded closed set of $\M$. This is explained in \Cref{closedcompact}. The inclusion principle then implies that the boundary of this flow on $\M$ projects well back to $M$. This is the content of \Cref{projectswell} and \Cref{diffeoclose}.  We emphasize that this is a general construction and up to that point there are no curvature assumptions on $S$.  Finally, the proof of \Cref{theoremHomology} is in \Cref{proofhomology}.

 \begin{figure}[h]
\centering
  \includegraphics[width=260pt]{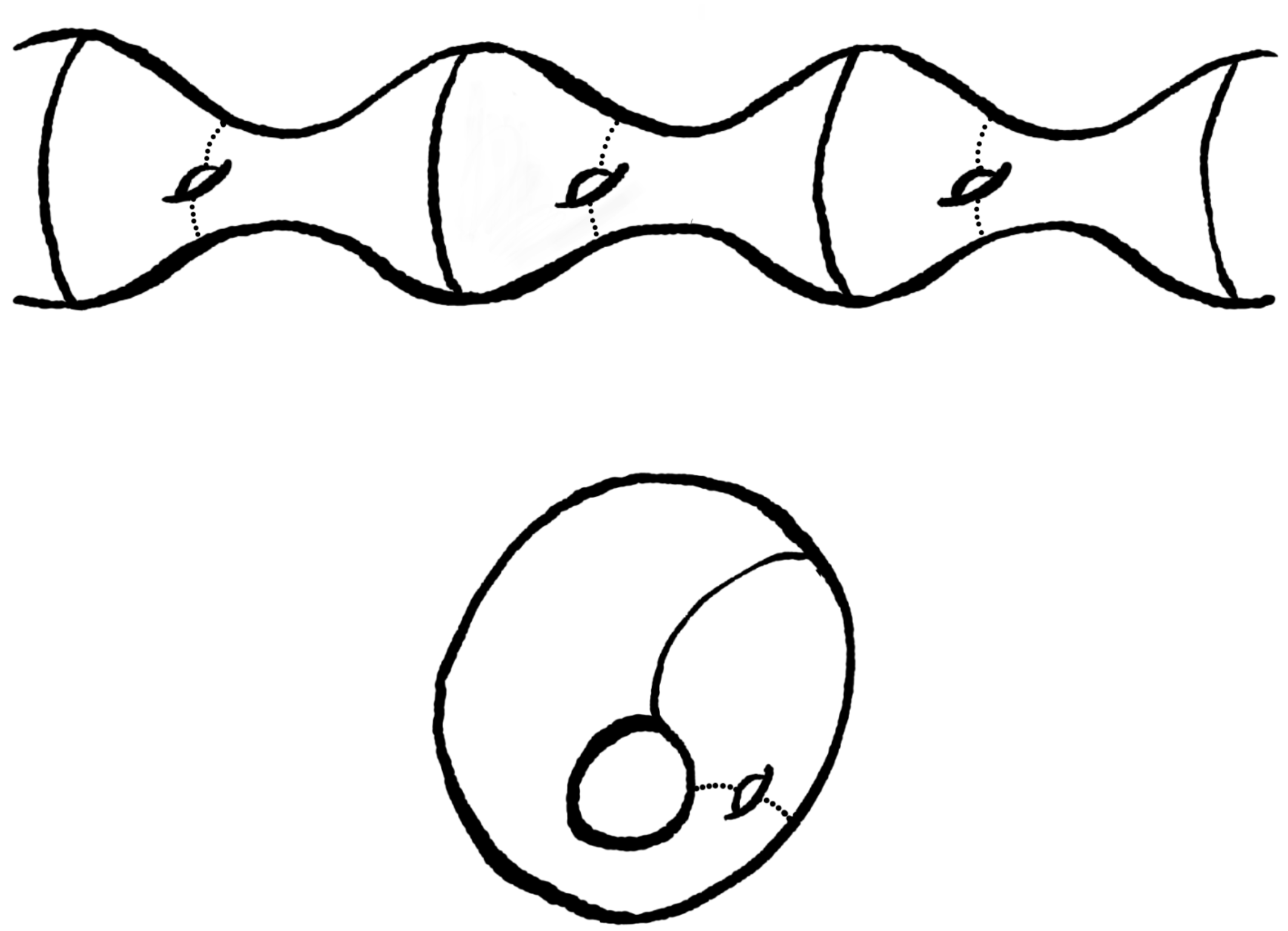}
  \caption{A covering space (top) that depends only on the homology class of the surfaces shown in the base space (bottom). The continuous line represents a surface with non-vanishing mean curvature and the dotted lines a disconnected area-minimizing hypersurface in the same homology class. For simplicity the surfaces in the figure are disjoint, but this does not have to be the case.}
  \label{figCovering}
 \end{figure}

A difference with the separating case is that components of $\Gamma$ might occur with high multiplicity. This can happen even when the initial condition has non-vanishing mean curvature. An example with multiplicity 3 is illustrated in \Cref{figmult3}. The same idea can be used to produce limits with arbitrarily high multiplicity. However, one can still  control the multiplicity of some components in some situations, e.g. \Cref{theoremIncompressible},\Cref{theoremNegative} or, if $M$ is irreducible then one automatically has that $\Gamma'$ has multiplicity one if $S$ is non-separating and incompressible.

\begin{figure}[h]

\hbox{\centering\hspace{5pt} \includegraphics[width=210pt]{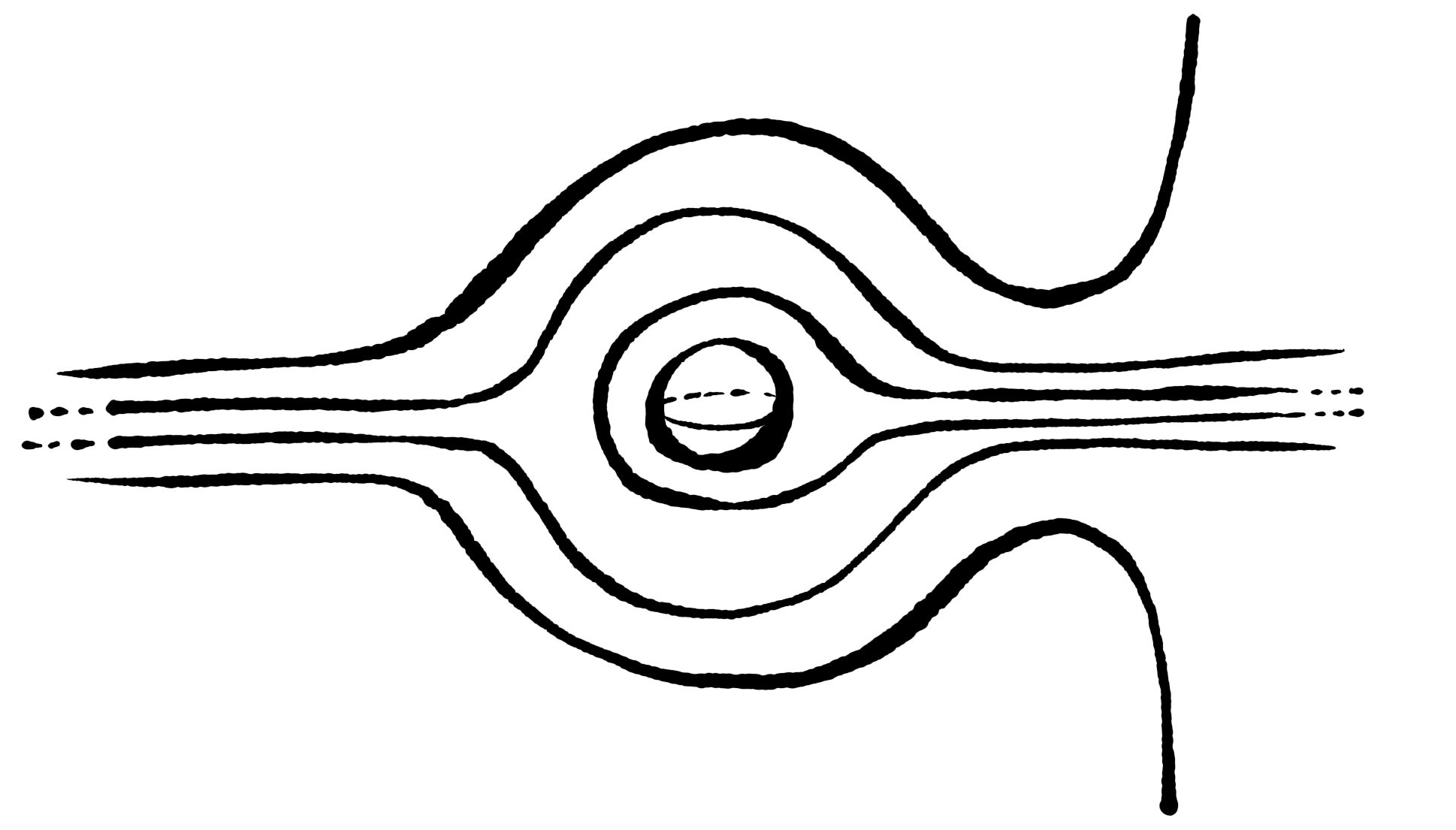}\hspace{0pt} \includegraphics[width=130pt]{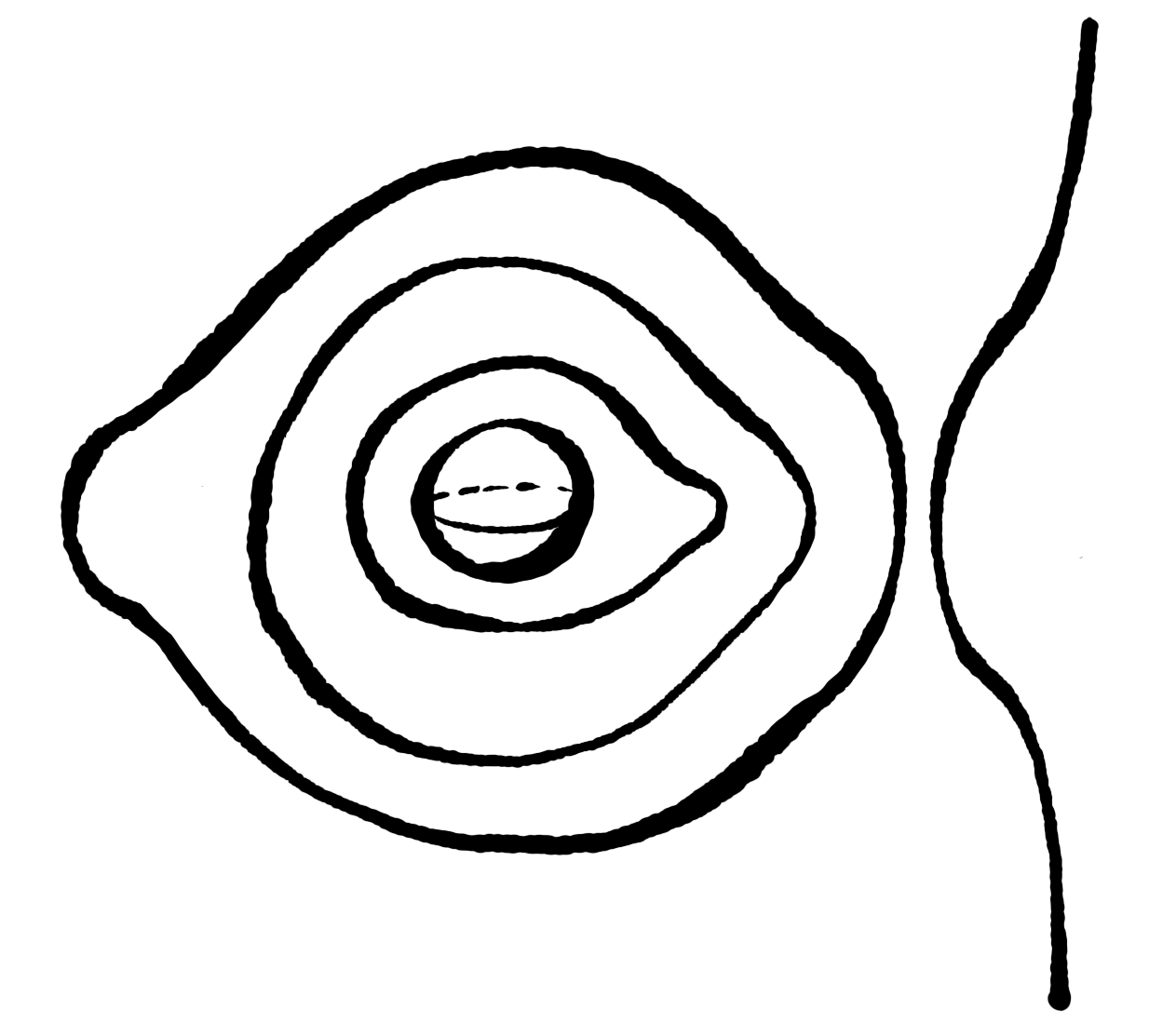}}
  \caption{The first figure (left) represents a connected surface with non-vanishing mean curvature that has a long tube going around the ambient manifold 3 times. Each time it surrounds the same minimal 2-sphere. Its evolution (right) has components accumulating with multiplicity 3 around this 2-sphere.}
  \label{figmult3}
 \end{figure}

Next, we study the mean curvature flow with surgery in dimension 2. This flow requires choosing parameters that control the dimensions of the surgeries and the possibility of discarding components after surgery. To prove \Cref{theoremSurgery}, we first notice that if these parameters are fine enough, then the flow with surgery can be continued smoothly towards $\Gamma$. This is the content of \Cref{theoremContinuation} (a similar result was proved in \cite{BrendleHuisken3m}). Then we use that the level set flow projects well (\Cref{diffeoclose}) in combination with a result of Lauer (\Cref{theoremLauer}) to conclude that surgeries also project well as long as the surgery parameters are fine enough. This is in \Cref{surgeryprojects}. Finally, the proof of \Cref{theoremSurgery} can be found in \Cref{proofSurgery}.

We then focus on the situation in which $S$ is incompressible. A key observation in the proof of \Cref{theoremIncompressible} is that despite surgery and discarding, the flow always has exactly one incompressible component (\Cref{incompressiblecomponents}). This is because, on one hand, neck surgeries correspond to compressing disks (\Cref{deftrivial}), and on the other, incompressible components are never discarded as a consequence of the canonical neighborhood theorem (\Cref{corollarycanonical}). 

In negative curvature, the flow of an incompressible surface only splits mean-convex $3$-balls which eventually become extinct. Haslhofer-Ketover \cite{HaslhoferKetover} constructed monotone mean-convex isotopies from such balls to a marble tree (\Cref{marbles}, left). To construct the isotopy from \Cref{theoremIncompressible} we first show that, instead of performing surgeries on the incompressible component, one can use the isotopy from \cite{HaslhoferKetover} to obtain an incompressible surface with marble trees attached to it (\Cref{marbles}, right). One then retracts the marble trees as in \cite{HaslhoferKetover}, first contracting balls in loose ends, then retracting the tubes attached to them into the next balls on the three, and so on. After retracting all the trees one can continue the flow normally until hitting a new surgery time. This is how we obtain a monotone isotopy between $S$ and $\Gamma$ in \Cref{theoremIncompressible}. Then \Cref{theoremNegative} is obtained after checking that the multiplicity conditions apply in negative curvature. 

\begin{figure}[h]
\hbox{\hspace{20pt} \includegraphics[width=130pt]{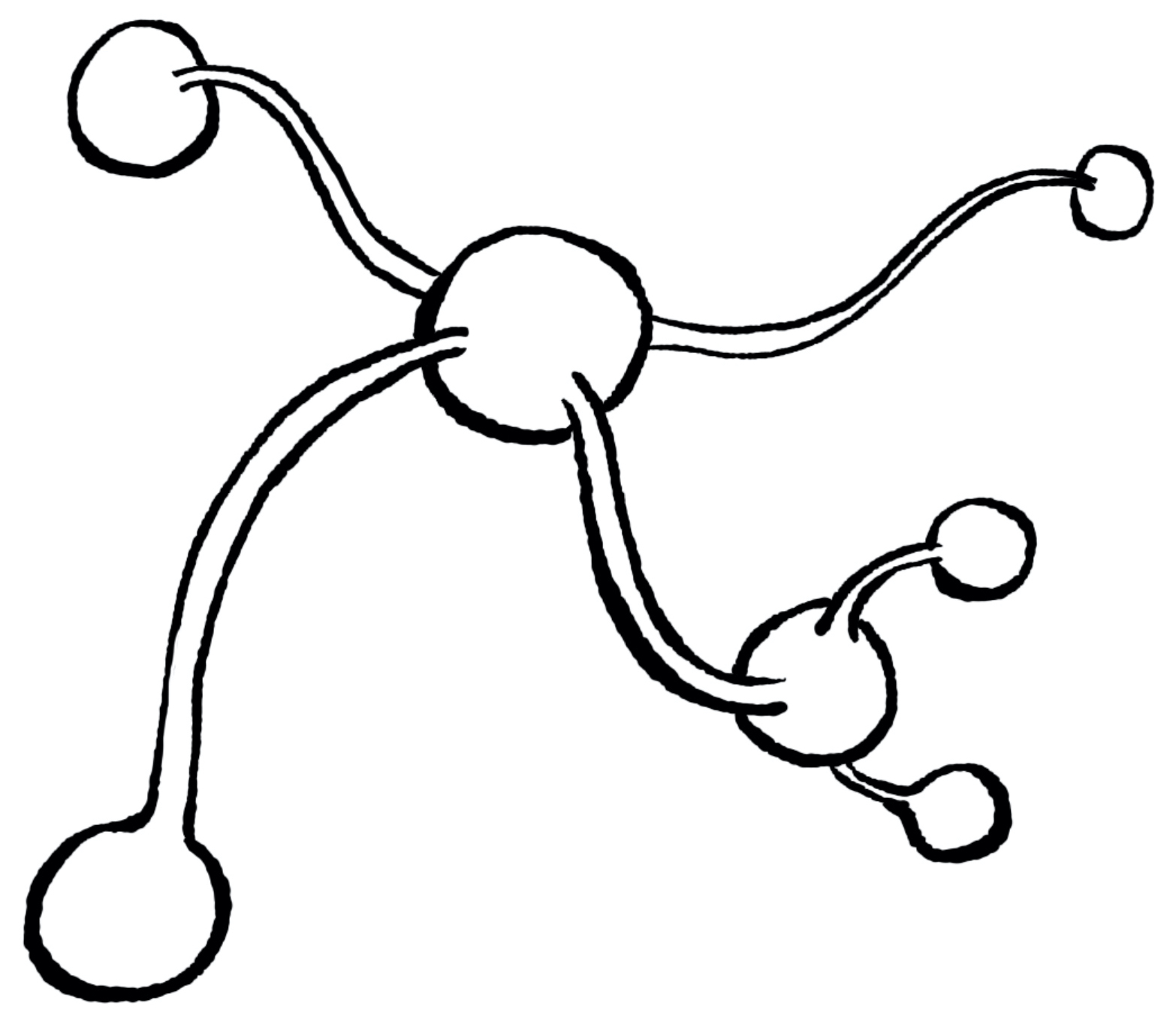}
\hspace{55pt} \includegraphics[width=130pt]{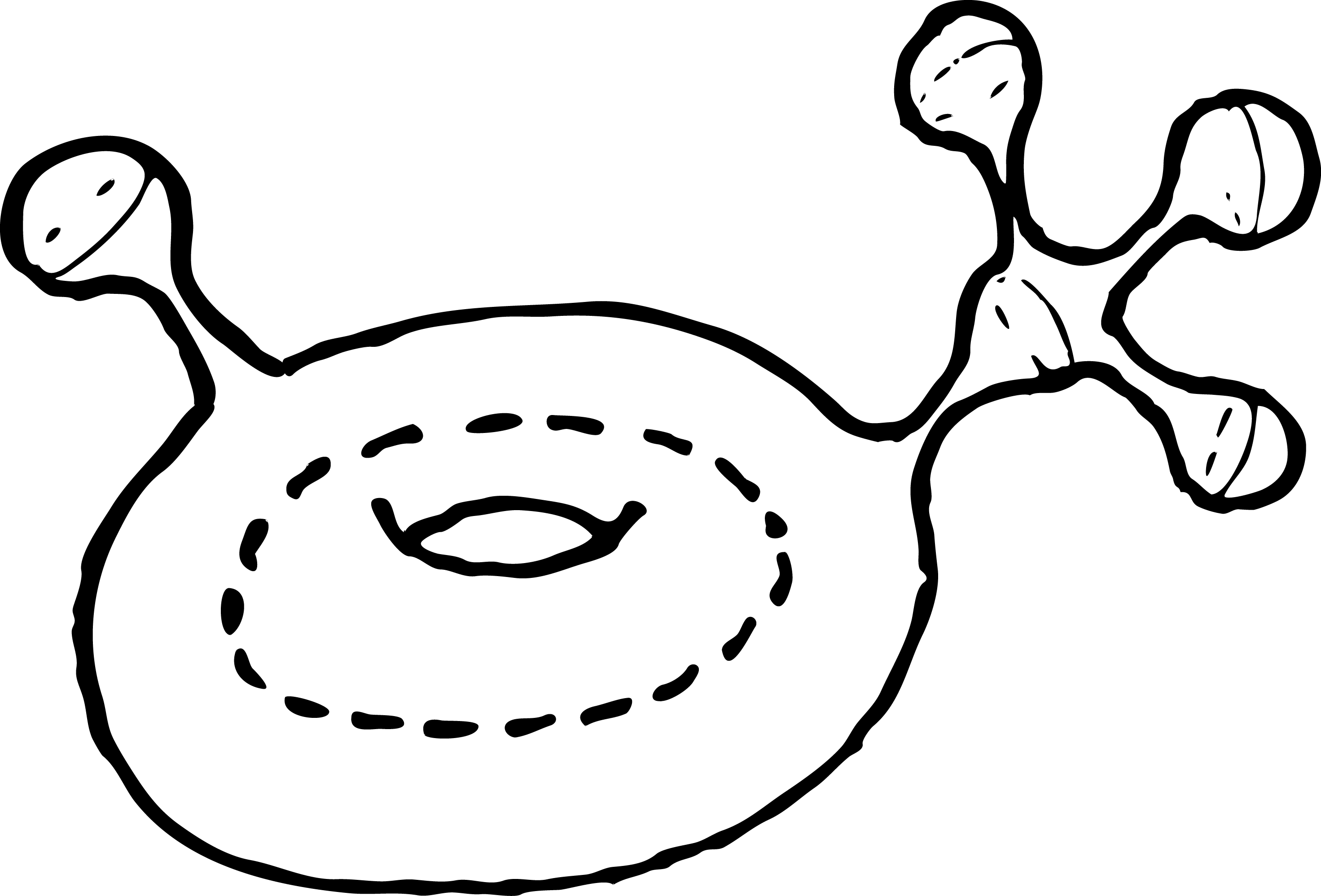}
}
  \caption{A marble tree (left) and an incompressible torus with marble trees attached (right)}
  \label{marbles}
 \end{figure}

Next, we construct global foliations on quasi-Fuchsian manifolds and hyperbolic manifolds that fiber over the circle. The idea is that one can always partition such manifolds by a maximal set of disjoint stable minimal surfaces in the isotopy class of $S\times \{t\}$. This partition can be further refined by adding minimal surfaces of index one via a min-max construction on each region bounded by two locally area-minimizing surfaces. Finally, one can use \Cref{theoremNegative} to foliate the region from an unstable minimal surface towards a stable one. 

\begin{figure}[h]
\centering \includegraphics[width=200pt]{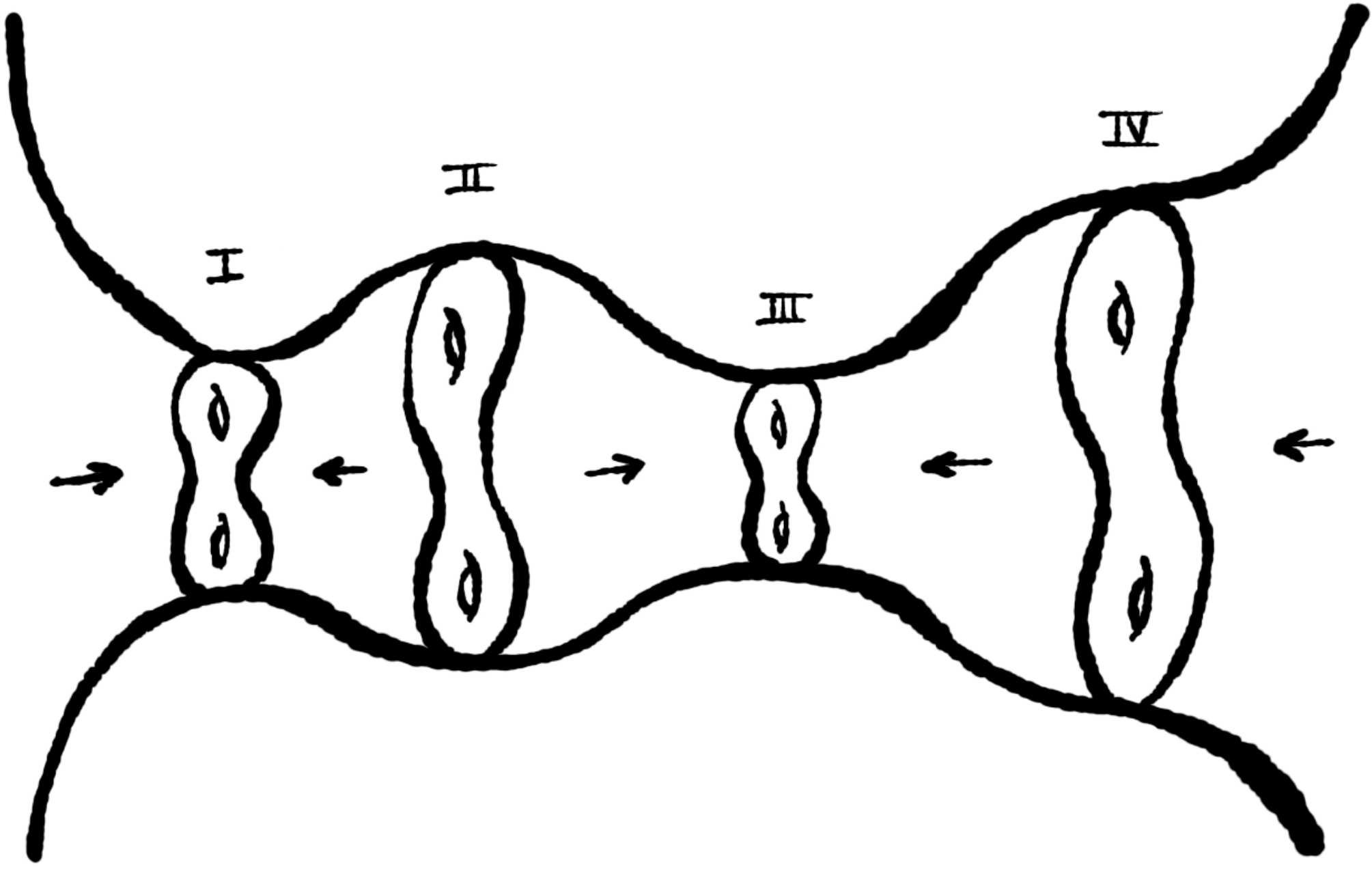}
  \caption{A quasi-Fuchsian manifold. I-IV are minimal surfaces. The arrows indicate the direction of the mean curvature vector of the foliation we construct on the complement of the minimal surfaces. I and IV are outermost. IV is not locally area-mininizing (we show that this is rare, but it happens). II is constructed via a min-max of sweeoputs joining I and III.}
  \label{figQF}
 \end{figure}

In the last section, we study the continuity of outermost minimal surfaces  in quasi-Fuchsian manifolds (I and IV in \Cref{figQF}) with respect to variations of the quasi-Fuchsian metric. These metrics are parametrized by the two conformal classes on $S$ at the boundary at infinity, modulo diffeomorphisms isotopic to the identity. This space is represented as the product of Teichmuller space $\mathcal{T}=\mathcal{T}(S)$ with itself and is diffeomorphic to $\R^{12g-12}$, where $g$ is the genus of $S$. We prove \Cref{theoremOutermost} in several stages in \Cref{secOutermost}. The main observation for the continuity is that a large set of quasi-Fuchsian metrics are what we call \textit{bumpy along the fiber}. This means that there are no degenerate minimal surfaces in the isotopy class of $S\times\{t\}$.
In addition, we derive the existence of discontinuities by combining the analysis from \cite{Uhlenbeck} with the theorem of invariance of domain. As a result we conclude that there are outermost surfaces which are stable but not locally-area minimizing (see surface IV in \Cref{figQF}).

\section{The level set flow}

Let $M$ be a closed manifold and $K\subset M$ a closed set. The \textit{level set flow} starting at $\Gamma$ (also called ``biggest flow'' in \cite{IlmanenShort}) consists of the maximal one-parameter family of closed sets $$t\mapsto F_t(K)$$ starting at $F_0(K)=K$ and satisfying the avoidance principle when compared with any compact classical mean curvature flow (see item (4) below). This property is also described by saying that $F_t(K)$ is \textit{the maximal set-theoretic subsolution of the mean curvature flow} starting at $K$ (see \cite[Subsections 4A and 4H]{IlmanenShort}). 

The following are standard properties of the level set flow. For their proofs, we refer the reader to: \cite{Ilmanen,IlmanenShort,WhiteSize}.

\subsection{Properties}\label{properties} Let $S,K,$ and $K'$ compact subsets of $M$, then:
\begin{enumerate}
\item $\cup_{t\geq 0}F_t(K)\times\{t\}$ is a closed subset of $M\times \R_{\geq 0}$.
\item If $S=\partial K$ is an embedded hypersurface then $\partial F_t(K)$ coincides with the classical mean curvature flow of $S$ up to its first singular time. 
\item $F_t(F_s(K))=F_{t+s}(K)$, for all $t,s\geq 0$.
\item \textit{Avoidance principle}: If $K\cap K'=\emptyset$, then $$F_t(K)\cap F_t(K')=\emptyset \text{ and}$$ $$F_t(K)\cup F_t(K')=F_t(K\cup K'),$$ for all $t\geq 0$.
\item \textit{Inclusion principle}: $K\subset K' \implies F_t(K)\subset  F_t(K')$, for all $t\geq 0$. 
\item \textit{Interior inclusion principle}: $K\subset \operatorname{int} K'\implies  F_t(K)\subset \operatorname{int}  F_t(K'),$ for all $t\geq 0$.
\item \textit{Finite speed}: Let $r_0>0$ and assume $\operatorname{Ric}\geq -C$ on $B(p,r_0)\subset M$. There exists $\tau=\tau(C)>0$, such that $$ F_{t_0}(K)\cap B(p,r_0)=\emptyset \implies F_t(K)\cap \{p\}=\emptyset,$$ for all $t\in[t_0,t_0+\tau]$.
\item \textit{Homology classes}: If $S$ is a non-separating embedded hypersurface, then $F_t(S)$ never vanishes.
\end{enumerate}

Following \cite{WhiteSize}, we say that a compact set $K\subset M$ is \textit{mean-convex} if $$F_t(K)\subset \operatorname{int} K$$ for all small $t$. This is the case if $\partial K$ is \textit{strictly mean-convex}, i.e. smooth and its mean curvature vector is never zero and always points towards the interior of $K$. 

A regularity and structure theory for the level set flow of mean-convex sets was developed by B. White \cite{WhiteSize,WhiteLocal,WhiteTopology} (see also \cite{HaslhoferKleiner1,HaslhoferSingular}). The following theorem summarizes some aspects from \cite[Theorem 1]{WhiteTopology} and \cite[Theorem 11.1]{WhiteSize}.
 
\begin{theorem}\label{theoremWhite}
Let $K \subset M$ be a compact set such that $\partial K$ is smooth and strictly mean-convex. Then
\begin{enumerate} 
\item $F_s(K)\subset \operatorname{int} F_t(K)$, for all $0\leq t< s$.
\item $F_t(\partial K)=\partial F_t( K)$, for all $t\geq 0$.
\item \textit{Outward minimizing property:} $\operatorname{Area}(\partial F_t(K))\leq \operatorname{Area}(\partial S)$, for all $ F_t(K)\subset S\subset \overline{S}\subset \operatorname{int}K$, and $t>0$.
\item $F_t(K)$ is smooth for a.e. $t \geq 0$ and $t$ large enough. 
\item The limit set $K_\infty=\cap_{t\geq 0}F_t(K)$ has finitely many connected components and its boundary $\partial K_\infty$ is a disjoint union of embedded stable minimal surfaces.  
\item As $t\to\infty$, $ F_t(\partial K)$ converges smoothly towards $\partial K_\infty$ with multiplicity $1$ or $2$. More precisely, if $C$ is a component of $K_\infty$, then the convergence towards $\partial C$ is one-sheeted if $\operatorname{int} C\neq \emptyset$ and two-sheeted if $\operatorname{int} C=\emptyset$.
\end{enumerate}
\end{theorem}

\section{The level set flow in homology}

In this section, we assume that $S\subset M$ is an oriented embedded hypersurface that is non-separating, i.e. $M\setminus S$ is connected. Our first goal is to find a canonical way of relating the evolution $F_t(S)$ to the level set flow of a boundary. As it turns out, this is possible with the aid of a covering space $\pi:\M\to M$ which is associated to the homology class of $S$. This covering space was previously used by Ilmanen in the proof of item (8) from Properties \ref{properties}, see \cite{Ilmanen}. We now construct this covering.

First, one thinks of $\overline{M}=M\setminus S$ as a compact manifold with boundary. Using the orientation of $S$ we divide the boundary components of $\overline{M}$ into two groups $\partial \overline{M}=S'\cup S''$. Components in $S'$ (resp. $S''$) are those with normal vector pointing towards the interior (resp. exterior) of $\overline{M}$. Notice that both $S'$ and $S''$ are diffeomorphic to $S$. Next, we label as $\overline{M}_i$, for $i\in\Z$, disjoint identical copies of $\overline{M}$. We form $\M$ by gluing together these pieces. More precisely:

\begin{definition}\label{coveringdef}
Given $S\subset M$ as above, we denote by $\M$ the connected manifold obtained from $\cup_{i\in \Z}\overline{M}_i$ by gluing together $S''_{i-1}$ with $S'_{i}$, for all $i\in \Z$. The projection given by the inclusion $\overline{M}_i\to M$ is denoted as $$\pi:\M\to M.$$ There is also a natural isometry given by the identity $\overline{M}_i \to \overline{M}_{i+1}$, which we denote as $$\rho: \M\to \M.$$ 
\end{definition} 

Before initiating the discussion on the level set flow, we need to establish some basic properties of $\M$. The following lemma will be useful for producing canonical lifts of disconnected hypersurfaces in the class $[S]$ (see \Cref{remarklift}).

\begin{lemma}\label{lifts}
The map $\pi:\M\to M$ is a regular covering space and $\rho$ generates its group of deck transformations. Moreover, $\M$ depends only on $[S]\in H_{n-1}(M,\Z)$. By this we mean that repeating the construction from \Cref{coveringdef} using any non-separating hypersurface $\Sigma\subset [S]$ yields the same $\M$, up to isometries. \end{lemma}

\begin{proof} Every class $[S]\in H_{n-1}(M,\Z)$ has an associated $\Phi_{[S]}\in H^1(M,\Z)$ via  Poincar\'e duality. The homomorphism $$\Phi_{[S]}:H_1(M,\Z) \to \Z$$ is the \textit{algebraic intersection number} of the class $[S]$. To compute $\Phi_{[S]}([\gamma])$, for $[\gamma]\in H_1(M,\Z)$, one first chooses smooth and transversal representatives $S$ and $\gamma$, so that $S\cap \gamma$ is a discrete set of points with signs determined by the orientations of $M$, $S$ and $\gamma$. It is a classical result that adding these signs together produces an integer $\Phi_{[S]}([\gamma])$ that depends only on the classes $[S]$ and $[\gamma]$, but not on the representatives $S$ and $\gamma$. 

Since $H_1(M,\Z)$ is the abelianization of $\pi_{1}(M)$, this induces an homomorphism $\phi_{[S]}:\pi_1(M)\to \Z$. The subgroup $\ker \phi_{[S]} \subset \pi_1(M)$ consists of exactly those closed loops that have intersection number zero with $[S]$, as 1-cycles.

\begin{claim*}
$\pi_*(\pi_1(\M))=\ker \phi_{[S]}.$
\end{claim*}

Let $\widetilde{\gamma}:[0,1]\to \M$ be a continuous path with end points on the same fiber, i.e. $\pi(\widetilde{\gamma}(0))=\pi(\widetilde{\gamma}(1))$, so that $\gamma=\pi(\widetilde{\gamma})$ is a closed loop on $M$.  Then, if $\widetilde{\gamma}(0) \in \overline{M}_{i_0}$ and $\widetilde{\gamma}(0) \in \overline{M}_{i_1}$, the intersection number $\phi_{[S]}([\gamma])$ is exactly $i_1-i_0$. For closed loops in $\M$, we must have $i_0=i_1$. This shows $\pi_*(\pi_1(\M))\subset \ker \phi_{[S]}$. For the same reason, if $\phi_{[S]}([\gamma])=0$, any lift of $\gamma$ is a closed loop in $\M$. This shows the other contention and proves the claim.

Finally, notice that $\ker \phi_{[S]}$ is a normal subgroup of $\pi_1(M)$, so it determines a unique regular covering $\M$ with deck transformation group $\pi_1(M)/\ker \phi_{[S]}\simeq\Z$. Since $\ker \phi_{[S]}$ depends only on $[S]$, the classification of covering spaces implies that repeating the construction above with another non-separating hypersurface in $[S]$ would produce the same covering space, up to an isometry. 
\end{proof}

\begin{remark}\label{remarklift}
When considering a lift $\widetilde{S}$ of a non-separating hypersurface $S$, we have to be careful where different connected components go to if we want lifts to be homologous on $\M$. The construction above gives us
$\pi^{-1}(S)$ as the disjoint union of $S''_{i-1}=S'_i$, $i\in\Z$, all diffeomorphic to $S$. From now on, we assume $$\widetilde{S}:=S'_0.$$ We also define the closed ends $$\widetilde{S}^+:=\cup_{i\geq 0} \overline{M}_i$$ and
$$\widetilde{S}^-:=\cup_{i< 0} \overline{M}_i.$$ If $\Sigma \in [S]$ is non-separating, then we define $\widetilde{\Sigma}$ as the lift given by the construction above when applied to $\Sigma$ and shifted through multiples of $\rho$ so that it intersects $M_0$ but not $M_i$, $i< 0$. We define the closed ends $\widetilde{\Sigma}^+$ and $\widetilde{\Sigma}^-$ accordingly, so that their boundary is $\widetilde{\Sigma}$.
\end{remark}

\begin{lemma}\label{minimalbarriers}
Let $\Sigma$ be a hypersurface minimizing the area in $[S]$. Then, $\Sigma$ is non-separating and lifts to $\widetilde{\Sigma}\subset \M$ in the sense of \Cref{remarklift}. Moreover, for $k\in\N$ large enough, $$X_{k}:=\rho^{-k}( \widetilde{\Sigma}^+)\cap \rho^{k}(\widetilde{\Sigma}^-)$$ is a smooth compact with minimal boundary and $$\widetilde{S}\subset \operatorname{int} X_k.$$
\end{lemma}

\begin{proof}
Although $\Sigma$ might be disconnected, no combination of its components forms a boundary. Otherwise we would produce a competitor for the area in the same homology class after deleting those components. Therefore, $M\setminus \Sigma$ is connected, i.e. $\Sigma$ is non-separating. The rest of the statement follows from directly from \Cref{remarklift} and \Cref{coveringdef}.
\end{proof}

In \cite{Ilmanen}, Ilmanen showed that the level set flow $F_t(S)$  does not vanish (see (8) in Properties \ref{properties}). The following result implies that its lift to $\M$ is restricted to a compact region. Roughly speaking, this can be interpreted as saying that the evolution of $S$ cannot spin around $M$ indefinitely. 

\begin{proposition}\label{projectswell}
 $\pi$ is a diffeomorphism on a neighborhood of $F_t(\widetilde{S})$. Moreover, $\pi(F_t(\widetilde{S}))=F_t(S)$ and $F_t(\widetilde{S})\subset \operatorname{int}X_k$.
\end{proposition}

\begin{proof}
The hypersurfaces $\rho^{i}(\widetilde{S})$, for $i\in \Z$, are disjoint by construction. Then, by the avoidance principle, $F_t(\rho^{i}(\widetilde{S}))=\rho^{i}(F_t(\widetilde{S}))$ are also disjoint. This implies that $\pi$ is a diffeomorphism on a neighborhood of $F_t(\widetilde{S})$. Therefore, its projection to $M$ is a level set flow, and by uniqueness $$\pi(F_t(\widetilde{S}))=F_t(S).$$  \Cref{minimalbarriers} implies $\widetilde{S}\subset \operatorname{int}X_k$ and by construction $X_k$ is a smooth compact region with minimal boundary, so $F_t(X_k)=X_k$. The last statement then follows from the interior inclusion principle.
\end{proof}

From the proof above and a simple compactness argument, we deduce the following observation that will be useful when proving that neck surgeries on $\M$ project well to $M$ as long as the neck scale is small enough (see \Cref{surgeryprojects}).

\begin{corollary}\label{diffeoclose}
For every $T>0$, there exists $\e>0$ such that $\pi$ is a diffeomorphism on a neighborhood of $\cup_{s\in[t,t+\e]}F_{s}(\widetilde{S})$, for all $0 \leq t \leq T.$
\end{corollary}

Another useful consequence of the minimal barriers from \Cref{minimalbarriers}, is that the level set evolution of certain non-compact closed sets behaves exactly like that of compact sets. More precisely:

\begin{proposition}\label{closedcompact}
Properties \ref{properties} and \Cref{theoremWhite} hold on the category of closed sets $K\subset \M$ such that $$\rho^{-i}(\widetilde{\Sigma}^+) \supset K \supset \rho^i(\widetilde{\Sigma}^+),$$ for some $i\in \N$.
\end{proposition}

\begin{proof}
The level set flow $F_t(K)$ of a closed set $K$ is well-defined, unique and exists for all time, as long as its boundary is a compact region (see \cite{IlmanenShort,Ilmanen}). This is the case in the class of closed sets under consideration. 

Since $\partial \widetilde{\Sigma}^+=\widetilde{\Sigma}$ is a smooth minimal hypersurface, we have $F_t(\widetilde{\Sigma}^+)=\widetilde{\Sigma}^+$. Let $X_i$ be as in \Cref{minimalbarriers}. From the definition of the level set flow, it is not hard to see that $F_t(K)\setminus \operatorname{int} \rho^{i}(\widetilde{\Sigma}^+)$ is a set-theoretic subsolution of the mean curvature flow, in the sense of \cite{IlmanenShort}, starting at $K\cap X_i$. In particular, $F_t(K)\setminus \operatorname{int} \rho^{i}(\widetilde{\Sigma}^+)\subset F_t(K\cap X_i).$ However, from the inclusion principle and the definition of $X_i$, we also have $F_t(K\cap X_i) \subset F_t(K)\setminus \operatorname{int} \rho^{i}(\widetilde{\Sigma}^+)$. Together, these imply $$F_t(K)=F_t(K\cap X_i) \cup \rho^i(\widetilde{\Sigma}^+).$$ In other words, the level set flow of $K$ is the union of a compact level set flow and a fixed closed set. The proposition now follows easily.
\end{proof}

This has consequences for the level set flow of non-separating hypersurfaces in $M$ with non-vanishing mean curvature vector. 

\subsection{Proof of Theorem \Cref{theoremHomology}}\label{proofhomology}
\begin{proof}
We choose the orientation of $S$ given by the mean curvature vector so that $\widetilde{S}^+$ is a smooth strictly mean-convex closed set. From \Cref{closedcompact} and \Cref{theoremWhite}, $F_t(\widetilde{S})=\partial F_t(\widetilde{S}^+)$ is smooth for a.e $t$ and $t$ large enough and converges to an embedded stable minimal hypersurface $\widetilde{\Gamma}\subset \M$. 
From \Cref{projectswell} and \Cref{diffeoclose}, this implies that $F_t(S)=\pi(F_t(\widetilde{S}))$ is smooth for a.e $t$ and $t$ large enough. 

We claim that $\Gamma=:\pi(\widetilde{\Gamma})$ is also an embedded minimal hypersurface. Remember $\widetilde{\Gamma}=\partial(\cap_{t\geq 0}F_t(\widetilde{S}^+)).$ Let $i<j\in \Z$, then $\rho^i(\widetilde{S}^+)\subset \rho^j(\widetilde{S}^+)$ and by the inclusion principle $$\rho^i(F_t(\widetilde{S}^+))\subset \rho^j(F_t(\widetilde{S}^+))$$ holds for all $t\geq 0$. Therefore, $$\rho^i(\cap_{t\geq 0}F_t(\widetilde{S}^+)) \subset \rho^{j}(\cap_{t\geq 0}F_t(\widetilde{S}^+)).$$ In particular, components of $\rho^i(\widetilde\Gamma)$ are contained in regions bounded by components of $\rho^j(\widetilde\Gamma).$ It follows that they are either disjoint or intersect tangentially, locally staying on one side from each other. In the latter case, the maximum principle for minimal hypersurfaces implies they must coincide. In other words, when we translate $\Gamma$ by multiples of $\rho$, intersections might occur, but they are always embedded (e.g. this situation is illustrated in \Cref{figmult3}). It follows that $$\cup_{i\in\Z}\rho^i(\widetilde\Gamma)$$ is an embedded hypersurface of $\M$. Therefore, so is $\Gamma=\pi(\widetilde{\Gamma})$.

Finally, to see that $F_t(S)$ does not vanish, notice that $\partial F_t(\widetilde{S}^+)$ does not vanish because $\rho^{k}(\widetilde{\Sigma}^+)\subset \widetilde{S}^+\subset \rho^{-k}(\widetilde{\Sigma}^+)$ (notice this also follows from item (8) of Properties \ref{properties} by a different argument, see \cite{Ilmanen}).

\end{proof}

\section{Mean-convex mean curvature flow with surgery}

We begin covering fundamental aspects of the mean-convex mean curvature flow with surgery for surfaces due to Haslhofer-Kleiner \cite{HaslhoferKleiner2} (another version of this flow was developed around the same time by Brendle-Huisken in \cite{BrendleHuiskenR3,BrendleHuisken3m}). Our presentation follows the one in \cite{HaslhoferKetover} which is made in the context of Riemannian manifolds (see also \cite{LiokumovichMaximo}).

The existence theory of \cite{HaslhoferKleiner2,HaslhoferKetover} is formulated using a class of flows called $(\alpha,\delta,\mathbb{H})$-flows.  We know explain these parameters. Roughly speaking, $\alpha$ is determined by the extrinsic geometry of the initial domain. The other two parameters can be chosen. The number $\delta$ denotes the fineness/regularity of the surgery necks. The triple $\mathbb{H}=(H_{trig},H_{neck},H_{th})$, with $H_{trig}>H_{neck}>H_{th}>1$, specifies when and where surgeries are performed. More precisely, when the maximum of the mean curvature reaches $H_{trig}$, surgeries are performed on a disjoint collection of necks with curvature $H_{neck}$. These necks separate the \textit{trigger part} $\{H=H_{trig}\}$ from the \textit{thick part} $\{ H=H_{th}\}$. After surgeries are performed components containing the trigger part are discarded. The flow is then able to continue until the maximum of the mean curvature reaches $H_{trig}$ once more. 

We present the definitions necessary to describe the surgeries performed on the flow with surgery of \cite{HaslhoferKetover}.

\begin{definition}\label{strongneck}(see Definition 7.2,  \cite{HaslhoferKetover})
Let $U$ be an open region of $M$. We say that a smooth family of smooth domains moving my mean curvature flow $\{K_t\subset U\}_{t\in I}$ has a \textit{strong $\delta$-neck} with center $p$ and radius $s$ at time $t_0$, if $4s/\delta \leq \operatorname{inj}(M)$ and the following condition holds:\\
\begin{itemize}
\item \emph{$\{s^{-1}\cdot \exp_p^{-1}(K_{t_0+s^2t}-p)\cap B_{2s/\delta}(p)\cap U\}_{t\in(-1,t_0]}$ is $\delta$-close in the $C^{[1/\delta]}(V)$-norm to the evolution of a solid round cylinder $D^{2}\times \R$ whose radius is equal to $1$ at $t=0$,}
\end{itemize}
where $V=B_{1/\delta}\cap \big[s^{-1}\cdot \exp_p^{-1}(B_{2s/\delta}\cap U)\big] \subset \mathbb{R}^3$.
\end{definition}

\begin{definition}\label{caps}(see Definition 2.8, \cite{BuzanoSpheres})
A \textit{standard cap} is a smooth convex domain $K^{st}\subset \mathbb{R}^3$ such that 
\begin{enumerate}
\item $K^{st}\cap\{x_1>0.01\}=B_1(0)\cap \{x_1>0.01\}$.
\item $K^{st}\cap\{x_1<-0.01\}$ is a solid round half-cylinder of radius 1.
\item $K^{st}$ is given by revolution of a function $u^{st}:(-\infty,1]\to \R_{+}.$
\end{enumerate}
\end{definition}

\begin{definition}\label{replacing}(see Definition 2.4 \cite{HaslhoferKleiner2} and Definition 6.2 \cite{BuzanoSpheres})
Let $K^-$ be the final slice of a smooth flow having a strong $\delta$-neck with center $p$ and radius $s$ contained on an open set $U$, and $K^\sharp \subset K^-$ a region with smooth boundary. We say that  \textit{$K^\sharp$ is obtained from $K^-$ by replacing the strong $\delta$-neck for a pair of standard caps}, if the following holds:
\begin{enumerate}
\item $K^-\setminus K^\sharp$ is contained in a ball $B=B(p,5\Gamma s)$.
\item there are bounds for the second fundamental form of $\partial K^\sharp$ and its derivatives:
$$\sup_{\partial K^\sharp\cap B}|\nabla^\ell A| \leq C_\ell s^{-1-\ell}, \ \ \forall\ \ell\geq 0.$$ 
\item If $B\subset U$, then for every point $p_\sharp \in \partial K^\sharp\cap B$, with $\lambda_1(p_\sharp)<0$, there is a point $p_-\in \partial K^- \cap B$ with $\frac{\lambda_1}{H}(p_-)\leq \frac{\lambda_1}{H}(p_\sharp)+\delta'(s)$.
\item If $B(p,10\Gamma s)\subset U$, then $s^{-1}\cdot \exp_p^{-1} (K^\sharp)$ is $\delta'(\delta)$-close in $B(0,10\Gamma)\subset \mathbb{R}^3$ to a pair of disjoint standard caps, which are at distance $\Gamma$ from the origin.
\end{enumerate} Where in (3) and (4), the function $\delta'$ is positive and $\lim_{x\to 0}\delta'(x)=0$. Whenever we are in the situation above, we refer to $K^-$ and $K^\sharp$ as the \textit{pre-surgery} and \textit{post-surgery} domains, respectively. We also say that the pair of standard caps are \textit{facing surgery caps}.   
\end{definition}

To define the mean curvature flow with surgery we must first introduce the notions of \textit{$\alpha$-Andrews flow} and \textit{$(\alpha,\delta)$-flow}.

\begin{definition}\label{alphaandrews}(see Definition 7.1, \cite{HaslhoferKetover})
Let $\alpha>0$. \textit{A smooth $\alpha$-Andrews flow} $\{K_t \subset U\}_{t\in I}$ in an open set $U\subset M$ of a closed Riemannian $3$-manifold $M$, is a smooth family of mean convex domains moving by mean curvature flow with $\inf H \geq 4\alpha/\operatorname{inj}(M)$, such that for every $p\in \partial K_t$, the two closed balls $\bar B^{\pm}(p)$ with radius $r(p)=\alpha/H(p)$ and center $c^\pm(p)=\exp_p(\pm r(p\nu(p))$, where $\nu$ is the inward unit normal of $\partial_t K$ at p, satisfy $\bar B^+(p)\cap U\subset K_t$ and $\bar B^+(p)\cap U\subset U\setminus \operatorname{Int}(K_t)$, respectively.
\end{definition}

\begin{definition}\label{alphadelta}
An $(\alpha,\delta)$-flow $\mathcal{K}$ is a collection of finitely many smooth $\alpha$-Andrews flow $\{K_t^i\subset U\}_{t\in[t_{i-1},t_i]}$ $(i=1,\dots,k-1)$ in an open set $U\subset M$ such that the following statements hold:
\begin{enumerate}
\item For each $i=1,\dots,k-1$, a finite collection of disjoint strong $\delta$-necks (Definition \Cref{strongneck}) of the final time slice $K_{t_i}^i = : K_{t_i}^-$ are replaced by pairs of standard caps (as described in Definition \Cref{replacing}). The resulting post-surgery domain is denoted by $K_{t_i}^\sharp \subset K_{t_i}^-$.
\item For $i\geq 1$, the initial time slice $K_{t_i}^{{i+1}}=:K^+_{t_i}$, is obtained from $K_{t_i}^\sharp$ by discarding some connected components.
\item There exists $s_\sharp=s_\sharp(\mathcal{K})>0$, which depends on $\mathcal{K}$, such that all the necks in item (1) have radius $s\in[\mu^{-1/2}s_\sharp,\mu^{1/2}s_\sharp]$ (here $\mu\in[1,\infty)$, see Convention 1.2 \cite{HaslhoferKleiner1}).
\end{enumerate}
\end{definition}

Finally, we present the definition of the mean-convex mean curvature flow with surgery:

\begin{definition}
An \textit{$(\alpha,\delta,\mathbb{H})$-flow}, $\mathbb{H}=(H_{th},H_{neck},H_{trig})$, is an $(\alpha,\delta)$-flow $\{K_t\subset U\}_{t\geq 0}$ such that the following statements hold.
\begin{enumerate}
\item $4\alpha/\operatorname{inj}(M)\leq H\leq H_{trig}$ everywhere, and surgery and/or discarding occurs precisely at times $t$ when $H=H_{trig}$ somewhere.
\item The collection of necks in the definition of $(\alpha,\delta)$-flow is a minimal collection of solid $\delta$-necks of curvature $H_{neck}$ which separate the set $\{ H=H_{trig}\}$ from $\{ H\leq H_{th}\}$ in the domain $K_t^{-}$.
\item $K_t^+$ is obtained from $K^\sharp_t$ by discarding precisely those connected components with $H>H_{th}$ everywhere. In particular, of each pair of facing surgery caps, precisely one is discarded. 
\item If a strong $\delta$-neck from item (2) also is a strong $\hat \delta$-neck for some $\hat\delta<\delta$, then item (4) of the surgery (i.e. replacing a strong $\delta$-neck with a pair of standard caps), also holds with $\hat \delta$ instead of $\delta$. 
\end{enumerate}
\end{definition}

For any smooth mean-convex domain $K_0$ which is mean-convex, there is a choice of parameters $\alpha$ so that the mean curvature flow for short time is an $(\alpha,\delta,\mathbb{H})$-flow. Curvature and convexity estimates, canonical neighborhood theorems, as well as the long time existence theory for $(\alpha,\delta,\mathbb{H})$-flows, are developed for $\mathbb{R}^3$ in \cite{HaslhoferKleiner2} and for general $3$-manifolds in \cite{HaslhoferKetover} (see also \cite{LiokumovichMaximo}). 

We recall the following existence result:

\begin{theorem}\label{mcfexistence} \cite[Theorem 7.7]{HaslhoferKetover}.
Let $K\subset M^3$ be a mean convex domain. Then for every $0<T<\infty$, choosing $\delta>0$ small enough and $H_{trig}\gg H_{neck}\gg H_{th}\gg 1$, there exists a $(\alpha,\delta,\mathbb{H})$-flow $\{K_t\}_{t\in[0,T]}$, with initial condition $K_0=K$.
\end{theorem}

The following Canonical Neighborhood Theorem implies a classification of discarded components:

\begin{theorem}\label{canonical}\cite[Theorem 7.6]{HaslhoferKetover}
For every $\e>0$ there exists $H_{can} <\infty$ such that if $\mathcal{K}$ is an $(\alpha,\delta,\mathbb{H})$-flow with $\delta$ small enough and $H_{trig}\gg  H_{neck}\gg H_{th}\gg1$, then every $(p,t)\in \partial \mathcal{K}$ with $H(p,t)\geq H_{can}$ is $\e$-close to either (a) an ancient $\alpha$-Andrews flow in $\R^3$ or (b) the evolution of a standard cap preceded by the evolution of a round cylinder $\overline{D}^2\times \R\subset \R^3$.
\end{theorem}

The following is a very useful consequence of the previous theorem:

\begin{corollary}\label{corollarycanonical}\cite[Corollary 1.25]{HaslhoferKleiner2} \cite[Corollary 3.4]{LiokumovichMaximo}
For $\e>0$ small enough, any $(\alpha,\delta,\mathbb{H})$-flow satisfying the hypothesis of \Cref{canonical} has discarded components diffeomorphic to the $3$-ball or solid torus $\overline{D}^2\times S^1$.
\end{corollary}

It is possible to approximate a level set flow by a sequence of flows with surgery (see \cite{Lauer, Head}). We recall the following result from \cite{Lauer} which we only state in dimension 2 (see also \cite[Proposition 1.27]{HaslhoferKleiner2}):

\begin{proposition}\cite[Claim in Proof of Theorem A]{Lauer}\label{theoremLauer} Let $K$ be a 3-dimensional compact mean-convex region. For each $T>0$, there exists a sequence of mean curvature flows with surgery $\{K_t^j,0\leq t \leq T\}$, with $K_0^j=K$ and a sequence $\e_j\to 0$, such that $F_{t+\e_j}(K)\subset K_t^j\subset F_t(K)$, for all $t \in 0\leq T$. At a surgery time this means $F_{t+\e_j}(K)\subset K_t^{j,+} \subset K_t^{j,\sharp}  \subset K_t^{j,-} \subset F_t(K)$. 
\end{proposition}

We now show that in a compact $3$-manifold, if the surgery parameters are fine enough, it is possible to continue a mean-curvature flow with surgery smoothly after finitely many surgeries so that it converges smoothly to a minimal surface. A similar result was proved in \cite[Theorem 1.2]{BrendleHuisken3m}.

\begin{theorem}\label{theoremContinuation}
Let $M$ be a $3$-manifold and $K\subset M$ a smooth compact region with mean-convex boundary. Assume that $$K_\infty=\cap_{t\geq 0}F_t(K)\neq \emptyset.$$ Then, there exists $T>0$ and a mean curvature flow with surgery $\{K_t,0\leq t\leq T\}$ such that $K_0=K$, $K_T$ is a regular time of the flow and $F_t(K_T)$ converges to $\Gamma=\partial K_\infty$ as $t\to\infty$, with the same multiplicity as $\lim_{t\to\infty}F_t(K)$.
\end{theorem}

\begin{proof}
Let $F_t:=F_t(K)$. From \Cref{theoremWhite} there exists some time $T>0$ after which $F_t$ is a classical smooth mean curvature flow and $F_t\to \Gamma$ smoothly, as $t\to\infty$. Fix $t_0\in(T,+\infty)$ and for each $j\in\N$, assume there exist a closed set $G^j$ with smooth and strictly mean-convex boundary $\partial G^j$, and $\e_j>0$, such that: \begin{enumerate}
\item $\e_j\to 0$, as $j\to \infty$,
\item $F_{t_0+\e_j}\subset G^j \subset F_{t_0}$ and
\item $\partial G^j$ is outward minimizing in $F_{t_0}$.
\end{enumerate}

First we show: 
\begin{claim*}
Let $\{G^j_t, t\geq t_0\}$ be the level set flow starting at $G^j_{t_0}:=G^j$. Then, for $j$ large enough, $\{\partial G^j_t, t\geq t_0\}$ is a classical mean curvature flow converging smoothly to $\Gamma$.
\end{claim*}

Since $G_{t_0}^j$ is strictly mean-convex, the boundary of the level set flow $\partial G^j_t$ is:
\begin{enumerate}
\item [(a)] a classical mean curvature flow on $[t_0,t_j)$, for some $t_j>t_0$.
\item [(b)] smooth, for a.e. $t>t_0$.
\item [(c)] outward minimizing inside of $G^j_{t_0}=G^j$, for all $t>t_0$.
\end{enumerate}
Combining (c) and (3) we can add:
\begin{enumerate}
\item [(d)] outward minimizing inside of $F_{t_0}$, for all $t>t_0$.
\end{enumerate}
By (2) and the inclusion principle, we have $F_{t+\e_j}\subset G_t^j\subset F_t$, for all $t\geq t_0$. For a.e. $t>t_0$, (b) holds for all $j\in\N$.  Then, at those times $\partial G^j_t$ is an outward minimizing family of smooth surfaces trapped in between the graphical convergence $\partial F_{t+\e_j}\to \partial F_{t}$. The standard argument using (d) to compare the area of $\partial G^j_t$ to that of the graphs (notice that $\partial F_{t+\e_j}$ is also outward minimizing), implies that the measures $\mathcal{H}^2\llcorner \partial G^j_t$ converge with multiplicity one to $\mathcal{H}^2\llcorner \partial F_t$ (see also \cite{Head,BrendleHuisken3m}). By continuity, the convergence extends to all $t\geq t_0$ and is uniform in compact sets of $[t_0,+\infty)$. If $\liminf_j t_j<+\infty$, then (\cite{WhiteLocal}, Theorem 3.5) implies $t_j$ is a regular time for some $j$ large, which is a contradiction. Therefore, we must have $\lim_j t_j=+\infty$. To conclude the claim, we want to show $t_j=\infty$, for $j$ large enough. In this case, we can apply the same argument as above to the smooth flow $\partial G_t^j$ on $[t_j-1,t_j)$, using the inclusion  $F_\infty \subset G^j_t \subset F_t$ and the fact that $\partial F_t\to \widetilde{\Gamma}_{\infty}$ smoothly as $t\to\infty$. This finishes the proof of the claim.

To prove the theorem, notice that from \Cref{theoremLauer} there exists a sequence of mean curvature flows with surgery $\{K_t^j,0\leq t \leq T+1\}$, and a sequence $\e_j\to 0$, such that $F_{t+\e_j}\subset K_t^j\subset F_t$, for all $t \in 0\leq T+1$. Choose $t_0\in(T,T+1)$ so that it is a regular time of $K^j$, for all $j\in\N$. Then $G^j=K_{t_0}^j$ satisfies the hypothesis of the claim. In particular, when $j$ is large enough, the mean curvature flow with surgery $K_t^j$ can be continued as a smooth flow on $t\geq t_0$ that converges to ${\Gamma}$ as $t\to\infty$. Finally, $F_{T+\e_j}\subset K_T^j\subset F_T$ implies $$F_{t+\e_j}\subset F_t(K^j_T)\subset F_t.$$ Therefore, the multiplicity of the convergence towards $\Gamma$ is the same as $F_t$. 
\end{proof}

Another consequence of \Cref{theoremLauer}, in combination with \Cref{diffeoclose}, is that the boundary of a flow with surgery on $\M$ projects well to a non-vanishing mean curvature flow with surgery on $M$.

\begin{theorem}\label{surgeryprojects}
Let $S$ a be non-separating surface with non-vanishing mean curvature vector on a closed $3$-manifold $M$. For each $T>0$, there exists $\{K_t,0\leq t\leq T\}$ a mean-convex flow with surgery starting at $K_0=\widetilde{S}^+$ such that $\pi$ is a diffeomorphism on neighborhoods of $\overline{K^-_{t_i}\setminus K^+_{t_i}}$ and $\partial K_t$, for $t_i$ and $t$ surgery and regular times, respectively. In particular, $\pi(\partial K^-_{t_i})$, $\pi(\partial K^\sharp_{t_i})$, $\pi(\partial K^+_{t_i})$ and $\pi(\partial K_t)$, define a flow with surgery on $M$.
\end{theorem}

\begin{proof}
From \cite[Theorem 3.2]{WhiteSize}, for all $\e>0$, we have $F_{t}(\widetilde{S}^+)\setminus F_{t+\e}(\widetilde{S}^+) =\cup_{s\in[t,t+\e)}F_s(\widetilde{S})$. If $\e>0$ is small enough, \Cref{diffeoclose} implies that $\pi$ is a diffeomorphism on $F_{t}(\widetilde{S}^+)\setminus F_{t+\e}(\widetilde{S}^+)$. Let $K^j_t$ be one fo the flows with surgery given by \Cref{theoremLauer}, for $\e_j<\e$. Then if $t_i$ and $t$ are singular regular times, respectively, we have 
$$F_{t_i+\e_j} (\widetilde{S}^+)\subset K^+_t\subset K^\sharp_t\subset K^-_t  \subset F_{t_i}(\widetilde{S}^+)$$
$$F_{t+\e_j} (\widetilde{S}^+)\subset K_t \subset F_{t}(\widetilde{S}^+)$$
and the theorem follows since $F_{t+\e} (\widetilde{S}^+) \subset \operatorname{int} F_{t+\e_j} (\widetilde{S}^+)$. 
\end{proof}

\subsection{Proof of \Cref{theoremSurgery}}\label{proofSurgery}

\begin{proof}
Given a non-separating surface  $S \subset M$ with non-vanishing mean curvature, let $\widetilde{S}^+ \subset \M$ be the region defined in \Cref{remarklift}. The result follows from applying \Cref{surgeryprojects}, \Cref{theoremContinuation} and \Cref{diffeoclose} to the evolution of $\widetilde{S}^+\subset \M$ .
\end{proof}

\section{Monotone isotopies of incompressible surfaces}

We begin with the following definition.

\begin{definition}\label{defisotopy}
Let $S$ and $\Gamma$ embedded surfaces of a $3$-manifold $M$. Assume that $S$ has non-vanishing mean curvature. A strictly monotone mean-convex isotopy between $S$ and $\Gamma$, is a smooth map $\phi:S\times[0,1]\to M$ such that 
\begin{enumerate}
\item $\phi|_{S\times\{t\}}$ is an embedding for all $0\leq t\leq 1$.
\item $\phi(S\times\{0\})=S$.
\item $\phi(S\times\{1\})=\Gamma$.
\item $\langle \vec{H}(p,t),(\partial \phi/\partial t)(p,t) \rangle>0$, for all $0\leq t<1$ and $p\in S$, where $\vec{H}(p,t)$ denotes the mean curvature vector of $\phi(S\times\{t\})$ at $\phi(p,t)$.
\end{enumerate}
\end{definition}

\begin{remark}\label{remarkArea} Notice that $\Gamma$ in \Cref{defisotopy} is allowed to be a minimal surface. Also, as a consequence of item (4) in \Cref{defisotopy} (and the first variation formula) the function $$t\mapsto \operatorname{Area}(\phi(S\times \{t\}))$$ is strictly decreasing on $t\in[0,1]$.
\end{remark}

In this section we construct strictly monotone mean-convex isotopies for incompressible surfaces with non-vanishing mean curvature in negatively curved $3$-manifolds. We begin by describing the effect that neck surgeries have in the topology of incompressible domains evolving by a mean curvature flow with surgery.

\begin{definition}\label{deftrivial} Let $\Sigma\subset M$ be a connected surface. We say that $\Sigma'$ is obtained from $\Sigma$ by a \textit{neck surgery}, if $\Sigma\setminus \Sigma'$ is a $2$-annulus, $\Sigma'\setminus \Sigma$ is a union of two disjoint embedded $2$-disks and $(\Sigma\setminus \Sigma')\cup(\Sigma'\setminus \Sigma)$ is a $2$-sphere bounding a $3$-ball. In particular, a neck surgery is a disk compression. 
\end{definition}

We have the following simple consequence of the definition:
\begin{corollary}\label{topoltrivial}
Let $S'$ be a surface obtained after performing a neck surgery on another surface $S$ (in the sense of \Cref{deftrivial}). If $S$ is incompressible, then $S'$ has two connected components: one is a $2$-sphere and the other is incompressible with the same genus as $S$. 
\end{corollary}

\begin{proof}
Let $c \subset \Sigma$ be one of the boundaries of the two 2-disks $\Sigma'\setminus \Sigma$. Since $c$ bounds a disk in $M$ (i.e. a compressible disk) and $\Sigma$ is incompressible, then $c$ also bounds a disk in $\Sigma$. The union of these two disks, which intersect only along $c$, is a $2$-sphere.
\end{proof}

Replacements of strong $\delta$-necks by standard caps, correspond to neck surgeries. In fact,  \Cref{caps}, \Cref{strongneck}, \Cref{replacing} and \Cref{alphaandrews}, imply that at every surgery time $t_i$, the region $\partial K_{t_i}^\sharp$ is obtained from $\partial K_{t_i}^-$ after performing finitely many neck surgeries in the sense of \Cref{deftrivial}.

\begin{definition}
We say that a smooth closed set $K$ is \textit{incompressible} if $\partial K$ is a \textit{connected} incompressible surface. Similarly, by the \textit{genus of $K$} we simply mean the genus of $\partial K$.
\end{definition}

Then, \Cref{topoltrivial} and \Cref{corollarycanonical} imply:

\begin{corollary}\label{incompressiblecomponents}
Let $\{K_t,0\leq t\leq T\}$ be a mean-convex mean curvature flow with surgery with $K_0$ incompressible. Then,
\begin{enumerate}
\item If $t$ is a surgery time, the sets $K^+_t \subset K^\sharp_t \subset K^-_t$ (see \Cref{replacing}) have exactly one incompressible component each, which we denote by $Q^+_t$,  $Q^\sharp_t$ and $Q^-_t$, respectively. In addition, $$Q^+_t=Q^\sharp_t\subset Q^-_t.$$
\item If $t$ is a regular time, the set $K_t$ has exactly one incompressible component, which we denote by $Q_t$.
\end{enumerate}   In particular, the incompressible component is never discarded. Any other component is bounded by a $2$-sphere. 
\end{corollary}

The idea for the construction of the isotopy in \Cref{theoremIncompressible} is to observe that if components bounded by $2$-spheres eventually vanish, then by \Cref{corollarycanonical} they correspond to $3$-balls. Then, one can use the gluing map from \cite[Theorem 4.1]{BuzanoSpheres} in combination with the isotopy constructed in \cite{HaslhoferKetover} to produce an isotopy between  $Q^-_t$ and $Q^\sharp_t$. 

The analytical tools for the construction of monotone isotopies are present in the works \cite{HaslhoferKetover}, \cite{BuzanoSpheres} and \cite{BuzanoTori}. We begin summarizing some of their definitions and results with a focus on how to bring the ideas to the context of incompressible surfaces. An important device in their work is the gluing map $\G$. Roughly speaking,  this map takes a small constant $r>0$, a mean-convex domain $K$ and a compact curve $
\gamma$ , with $\partial \gamma\subset \partial K$, and produces a new mean-convex domain $\G_r(K,\gamma)$ which is the result of gluing $K$ with a tubular neighbourhood of $\gamma$ of radius $r$. Both $K$ and $\gamma$ might have multiple finitely many components. As in  \cite{HaslhoferKetover}, we do not present the definition of the gluing map here, but refer the reader to \cite[Theorem 4.1]{BuzanoSpheres}

\begin{definition}
A \textit{marble tree} is a domain of the form $\G_r(D,\gamma)$, where:
\begin{enumerate}
\item $D=\cup_i B_i$ is a union of finitely many disjoint closed balls of the same radius,
\item $\gamma$ is a compact embedded curve with finitely many components
\item $(\gamma\setminus \partial\gamma)\cap D=\emptyset$,
\item $\partial \gamma \subset \partial D$,
\item $D\cup \gamma$ is simply connected.
\end{enumerate}
\end{definition}

The following result follows from the proof of \cite[Theorem 8.1]{HaslhoferKetover}:
\begin{theorem}\label{balltree}
Let $B\subset M^3$ be a smooth closed $3$-ball with strictly mean-convex boundary. If $B_\infty=\cap_{t\geq 0}F_t(B)=\emptyset$, then $B$ is isotopic to a marble tree $\G_r(D,\gamma) \subset B$ via a strictly monotone mean-convex isotopy. 
\end{theorem}

We will also need \cite[Corollary 6.4]{HaslhoferKetover}. This result allows us to join disconnected isotopies for a short time instead of performing surgeries or discarding components: 
\begin{proposition}\label{firstisot} \cite[Corollary 6.4]{HaslhoferKetover} Let $K^\sharp\subset M^3$ the result of performing surgeries over $K^-$ on a disjoint collection of $\delta$-necks. Assume that $\{K_t^\sharp\}$ is a strictly monotone mean-convex evolution of $K^\sharp=K_0^\sharp$. If $\gamma$ be the union of the almost straight lines connecting the tips of the opposing standard caps of each surgery, we denote by $\{\gamma_t\}$ the family of curves which follows $K_t^\sharp$ by normal motion starting at $\gamma$. Then, if $\delta$ is small enough, there exists a small perturbation of $\{\gamma_t\}$, which we also call $\{\gamma_t\}$, and $r$  and $\bar t$ small enough such that $K^-$ is isotopic to $\G_r(K^\sharp_{\bar t},\gamma_{\bar t})$. The isotopy is strictly monotone and mean-convex.
\end{proposition}

Following \Cref{incompressiblecomponents}, at a surgery time we can decompose $Q^-_t$ into the union of $Q^\sharp_t$ with some necks and regions bounded by $2$-spheres. More precisely, assume there are $k$ surgeries intersecting $Q^\sharp_t$, which we label using $j=1,\dots, k$. Each one of these surgeries has two associated closed sets: a $\delta$-neck $N^{j}_t$, which is simply the $3$-ball from \Cref{deftrivial}, and a smooth mean-convex region $B^{j}_t\subset K^{\sharp}_t$ spliting out of the neck and that is bounded by a $2$-sphere. Therefore, we have the decomposition \begin{equation}\label{partition}Q^-_t=Q^\sharp_t\cup (\cup_{j=1}^k N^{j}_t)\cup (\cup_{j=1}^k B^{j}_t).\end{equation}
 
\begin{remark} To produce the decomposition from \cref{partition} we have ignored all surgeries of $K^-_t$ disjoint from $Q^\sharp_t$. Notice that, $B_t^{j}\cap N^{j}_t$ and $Q_t^{\sharp}\cap  N^{j}_t$ are non-empty smooth 2-disks, for all $j$, and $B_t^{j}\cap (N^{j'}_t \cup B^{j'}_t \cup Q^\sharp_t)=\emptyset$, for all $j\neq j'$. Moreover, if $B^{j}_t$ it becomes extinct, \Cref{corollarycanonical} implies that these region is a $3$-ball.
\end{remark}

\begin{proposition}\label{ourisot} Let $Q_t^-,Q_t^\sharp$ and $Q_t$ be the incompressible components from \Cref{incompressiblecomponents}. Let $t_i$ be a surgery time and $$Q^-_{t_i}=Q^\sharp_{t_i}\cup (\cup_{j=1}^k N^{j}_{t_i})\cup (\cup_{j=1}^k B^{j}_{t_i})$$ the decomposition from \cref{partition}. If the level set flow of $B^j_{t_i}$ vanishes in finite time, then for $\delta$ small enough, there exists a strictly monotone mean-convex isotopy starting at $\partial Q_{t_i}^-$ and ending at $\partial Q_{t_i+\e}$, for all $\e$ small enough.\end{proposition}

\begin{proof}
We follow the ideas of \cite[Claim 6.6]{HaslhoferKetover} and \cite[Theorem 8.1]{HaslhoferKetover}. To simplify the notation, we avoid the subindex $t_i$ and  assume there is only one surgery, i.e. $$Q^-=Q^\sharp \cup N \cup B.$$ The general case follows along the same lines. 

From \Cref{balltree}, there exists a strictly monotone mean-convex isotopy $
\{B_t\}_{t\in[0,1]}$, starting at $B_0=B$ and ending on a marble tree $B_1.$ Denote by $\beta$ the almost straight line between the tips of the pair of standard caps of the surgery. Let $\beta_t$ the result of extending $\beta$ by following the two points $\beta\cap B_0$ and $\beta\cap Q_{t_i+t}$ along the evolutions $B_t$ and $Q_{t_i+t}$. By \Cref{firstisot}, $Q^-$ is isotopic to $\G_r(Q_{t_i+\bar t}\cup B_{\bar t},\beta_{\bar t})$, for some small $r$, via a strictly monotone mean-convex isotopy. As in the proof of \cite[Claim 6.6]{HaslhoferKetover} one can choose  a slowly decreasing positive function $r(t)>0$, with $r(\bar t)=r$, and a small slowly increasing function $\bar r(t)$, with $\bar r(\bar t)=t_i+\bar t<t_i+\e/2=\bar r(1)$, so that $\G_{r(t)}(Q_{\bar r(t)}\cup B_{t},\beta_{t})$ continues the isotopy after $\bar t$ with the desired monotonicity and mean-convexity. The extreme points of $\beta_{ t}$ now follow the evolutions of $Q_{\bar r(t)}$ and $B_{t}$ until the we reach $\G_{r(1)}(Q_{t_i+\e/2}\cup B_{1},\beta_{1})$, where $r$. This domain is the result of gluing  $Q_{t_i+\e/2}$ with the marble tree $B_{1}$ by attaching a narrow tube surrounding $\beta_1$. As in the proof of \cite[Theorem 6.1]{HaslhoferKetover}, one can use a variant of \cite[Section 5]{BuzanoSpheres} to shrink the marble tree monotonically into a small tube attached to $Q_{t^*}$, for some $t_i+\e/2<t^*<t_i+\e$. While doing so one keeps evolving $Q_{t}$ after time $t_i+\e/2$ as slow as necessary. Finally, one pushes the tube into $Q_{t}$, before it reaches $t=t_i+\e$.

\end{proof}

We can now prove the main result of this section:

\subsection{Proof of \Cref{theoremIncompressible}}\label{proofIncompressible}
\begin{proof}

Items (a1), (a2) and (b) follow from \Cref{incompressiblecomponents}, \Cref{corollarycanonical}, \Cref{theoremSurgery} and item (6) of  \Cref{theoremWhite}. The only point that remains to be shown is the existence of the isotopy.

We begin with the separating case. Let $S=\partial K$ where $K$ is a smooth mean-convex set. From the proof of \Cref{theoremContinuation} there exists $T>0$ such that, for all $\e>0$, there is a flow with surgery $$\{K_t,0\leq t\leq T\}$$ such that:
\begin{enumerate}
\item $K_0=K$,
\item $F_{t+\e}(K) \subset K_t\subset F_t(K)$, for all regular times $0\leq t\leq T$,
\item $F_{{t_i}+\e}(K) \subset  K^+_{t_i}\subset K^\sharp_{t_i}\subset K^-_{t_i}\subset F_{t_i}(K)$, for every surgery time $t_i$,
\item $F_t(K_T)$ is a classical flow for all $t\geq 0$,
\item $F_t(K_T)\to \Gamma$, $t\to\infty$ where components converge as either a one-sheeted or two-sheeted graph.
\end{enumerate}

For simplicity, for all $t\geq T$, we denote $$K_{t}:=F_{t-T}(K_T).$$ 

Denote by $Q^-_{t_i}$, $Q^\sharp_{t_i}$ and $Q_t$, the incompressible components from \Cref{incompressiblecomponents} for surgery and regular times $t_i$ and $t$, respectively.

\begin{claim*}
$K_0=Q_0$ and $K_T=Q_T$.
\end{claim*}

For $K_0=Q_0$ notice we are assuming that $S$ is connected and incompressible. To see $K_T=Q_T$, observe that $K_T$ by \Cref{incompressiblecomponents} if $K_T\setminus Q_T$ was not empty, then $K_T$ would have components bounded by $2$-spheres. Then, by (5) and classical degree theory, $\Gamma$ would contain either a copy of $\mathrm{S}^2$ or $\mathbb{RP}^2$, which contradicts our assumptions.

Now, let $$Q^-_{t_i}=Q^\sharp_{t_i}\cup (\cup_{j=1}^k N^{j}_{t_i})\cup (\cup_{j=1}^k B^{j}_{t_i})$$ be the decomposition from \cref{partition}. 

\begin{claim*}
$\cap_{t\geq 0} F_t(B^{j}_{t_i})=\emptyset$.
\end{claim*}

Let $B_\infty=\cap_{t\geq 0} F_t(B^{j}_{t_i})$ and $K_\infty=\cap_{t\geq 0}F_t(K)$. On one hand we have $B^{j}_{t_i}\subset Q^{-}_{t_i} \subset K^-_{t+i}\subset  F_{t_i}(K),$ from which it follows that $B_\infty \subset K_\infty$. From the previous claim and the definition of $K_t$, we also have $K_\infty= \cap_{t\geq 0}F_t(Q_T)\subset Q^{\sharp}_{t_i}.$
On the other hand $B_\infty\subset B^{j}_{t_i}$. Therefore, $B_\infty\subset  B^{j}_{t_i} \cap Q^{\sharp}_{t_i} = \emptyset$ (see \cref{partition}).

\begin{claim*}
For all $\bar\e>0$ small enough, there exists a monotone mean-convex isotopy connecting $\partial Q^-_{t_i}$ with $\partial Q_{t_i+\bar \e}$.
\end{claim*}

This follows directly from \Cref{ourisot}, which we can apply because of the previous claim.

\begin{claim*}
If $\operatorname{int} K_\infty \neq \emptyset$, then $F_t(\partial Q_T)$, for $t \in [0,\infty]$, is a monotone mean-convex isotopy between $\partial Q_T$ and $\Gamma$. If $\operatorname{int} K_\infty =\emptyset$, then $F_t(\partial Q_T)$, for $t \in [0,\infty]$, is a monotone mean-convex isotopy between $\partial Q_T$ and the boundary of a tubular neigborhood of $\Gamma$.
\end{claim*}

This follows immediately from \Cref{theoremWhite}.

The previous claims prove the result in the separating case. When $S$ is non-separating we can make $K=\widetilde{S}^+$. In this case, we automatically have $\operatorname{int} K_\infty \supset \rho^{k}(\widetilde{\Sigma}^+)\neq \emptyset$. In particular the convergence can only happen with multiplicity one. Moreover, by \Cref{surgeryprojects} and \Cref{projectswell} the isotopies between $\partial Q_{t_i+\e}$ and $\partial Q^-_{t_{i+1}}$, and between $\partial Q_T$ and $\widetilde{\Gamma}$ project diffeomorphically to $M$. 

It remains to see that the isotopies connecting $\partial Q^-_{t_i}$ and $\partial Q_{t_i+\bar \e}$ exist and project well. For the existence it is enough to check that in fact $B_{t_i}^j \subset \M$ is a $3$-ball. This follows from \Cref{corollarycanonical}. 
\begin{claim*}
$\pi$ is injective on each $B_{t_i}^j$. Moreover, $B_{t_i}^j \subset \M$ is a $3$-ball.
\end{claim*}

Notice that $\pi$ is injective in $\partial B_{t_i}^j \subset \partial K^\sharp_{t_i}\subset F_{t_i}(K)\setminus F_{t_i+\e}(K)$. From \Cref{incompressiblecomponents} we know that $\partial B_{t_i}^j \subset \M$ is a $2$-sphere. Therefore, $\partial B_{t_i}^j$ is a lift of the $2$-sphere $\pi(\partial B_{t_i}^j) \subset M$.  $\pi(\partial B_{t_i}^j)$ bounds a $3$-ball, which, as a simply connected set, lifts to $\M$. The lift of this $3$-ball is a region bounded by $\partial B_{t_i}^j$ in $\M$. So it is either, $B_{t_i}^j$ or $\M\setminus B_{t_i}^j$. However, since $\M$ is not compact, only the first case is possible and the claim follows.

Finally, to see that the isotopies project well, by monotonicity, it is enough to show that $\pi$ is a diffeomorphism on $Q^-_{t_i}\setminus Q_{t_i+\e}$. Notice that $$Q^-_{t_i}\setminus Q_{t_i+\e}=\Omega \cup (\cup_{j=1}^k \operatorname{int }B^{j}_{t_i}),$$
where $\Omega_0=(Q^\sharp_{t_i} \setminus Q_{t_i+\e}) \cup (\cup_{j=1}^k N^{j}_{t_i})\cup (\cup_{j=1}^k \partial B^{j}_{t_i})$ is a connected set. Since $\Omega_0\subset F_{t_i}\setminus F_{t_i+\e}$, it follows that $\pi$ is injective in $\Omega_0$. 

\begin{claim*}
Assume that $\pi$ is injective on subsets $\Omega$ and $B$ of $\M$. If $\Omega$ is connected, $B$ closed and $\partial B \subset \Omega$, then $\pi$ is injective on $B\cup \Omega$. 
\end{claim*}

Let $\rho'$ be a non-trivial deck transformation. Since $\pi$ is injective on $\Omega$ and $\partial B\subset \Omega$, we must have $\rho'(\Omega)\cap \partial B= \emptyset$. If $\rho'(\Omega)\cap B\neq \emptyset$, then $\rho'(\Omega)\subset \operatorname{int} B$, because $\rho'(\Omega)$ is connected. But this implies that $\rho'(\partial B)\cap B\neq \emptyset$ which contradicts our assumption of $\pi$ being injective on $B$.

We obtain that $\pi$ is injective on $Q^-_{t_i}\setminus Q_{t_i+\e}$ by applying the claim inductively. First to $\Omega_0$ and $B_{t_i}^1$ to conlude that $\pi$ is injective on $\Omega_1=\Omega_0 \cup B_{t_i}^1$. Then to $\Omega_1$ and $B_{t_i}^1$, and so on. This concludes the proof of the theorem.

\end{proof}

\subsection{Proof of \Cref{theoremNegative}}

\begin{proof}
As discussed in the first section, negatively curved 3-manifolds do not admit minimal 2-spheres or $\mathbb{RP}^2$. In the notation of \Cref{theoremIncompressible} this implies $\Gamma\setminus \Gamma'=\emptyset$. Therefore, in order to apply the last statement of \Cref{theoremIncompressible} we just need to check (a1), i.e. $\Gamma$ is obtained with multiplicity 1. 

First consider the case of $S$ non-separating and $M$ closed. Following the notation of \Cref{remarklift}, denote $\{K_t,0\leq t\leq T\}$ the flow starting at $K_0=\widetilde{S}^+$ given by \Cref{theoremContinuation}. Since $\Gamma\setminus \Gamma'=\emptyset$, from \Cref{incompressiblecomponents} the set $K_T$ consist of exactly one incompressible component $Q_T$. Since $\rho^{k}(\widetilde{\Sigma}^+)\subset Q_T$ (see \Cref{minimalbarriers}), it follows that $\operatorname{int} K_\infty =\operatorname{int} \cap_{t\geq 0}F_t(Q_T)\neq \emptyset$. From \Cref{theoremWhite}, $\widetilde{\Gamma}$ is is obtained with multiplicity one. Since $\widetilde{\Gamma}$ is connected it projects diffeomorphically to $\Gamma$, which is then obtained with multiplicity 1. 

Now, we assume that $S\subset \partial K$, $S$ is mean-convex and $\partial K\setminus S$ is a non-empty union of mean-convex and minimal surfaces. By the mean-convexity and minimality of the boundary we have $F_t(K)\subset K$, for all $t\geq 0$. Since we are only interested in the evolution of $S$, we can, without loss of generality, modify $K$ to ``close'' all of the other boundary components of $\partial K$, by attaching caps on them. Repeating the previous argument, after finitely many surgeries we are only left with exactly one boundary component which is incompressible. As before, the interior of the limit region is not empty. In this case, this is because the flow of $S$ can never reach $\partial K\setminus S$. This imply the convergence happens with multiplicity 1.

\end{proof}

\section{Foliations in hyperbolic $3$-manifolds}

We begin by recalling the following results due to Anderson.

\begin{theorem}\label{theoremAnderson1} \cite[Theorem 5.5]{Anderson} Let $N^3$ be a compact oriented 3-manifold with an analytic Riemannian metric. Then, either 
\begin{enumerate}
\item $N^3$ contains only finitely many compact stable, oriented, minimal surfaces of uniformly bounded area, or
\item $N^3$ fibres over $S^1$ with fibres smooth compact minimal surfaces
\end{enumerate}
\end{theorem}

\begin{theorem}\label{theoremAnderson2} \cite[Corollary 5.6]{Anderson}
A quasi-Fuchsian manifold contains at most finitely many closed stable minimal surfaces of a fixed genus.
\end{theorem}

\begin{remark}
In \cite{Anderson}, \Cref{theoremAnderson1} and \Cref{theoremAnderson2} are stated for a class of surfaces called $R$-locally area-minimizing. However, inspection of the proofs reveals that the crucial fact is having curvature estimates on the class of surfaces under consideration. This is the case for stable minimal surfaces by the work of Schoen \cite{Schoen}.
\end{remark}

The following lemma is standard and can be proved as in \cite[Theorem 2.3]{AnderssonCaiGalloway}. 

\begin{lemma}
Let $S \subset M$ be a stable minimal surface. Then, there exists a diffeomorphism $\phi: S\times [-\delta,\delta]\to M$ such that:
\begin{enumerate} 
\item $\phi|_{S\times\{0\}}$ is the identity, and 
\item $\phi(S\times\{t\})$ is either a minimal surface or has non-vanishing mean curvature.
\end{enumerate}
\end{lemma}

\begin{definition}\label{defcontracting}
Assume that $\phi(S\times\{t\})$ has non-vanishing mean curvature, for all $0<|t|<\delta$. We say that the closed halved tubular neigborhood $\phi(S\times[0,\delta])$ (resp. $\phi(S\times[-\delta,0])$)  is
\begin{enumerate}
\item \underline{contracting}: if the mean curvature vector of $\phi(S\times\{t\})$ points \underline{towards} $S$, for all $0<t\leq\delta$ (resp.  for all $ -\delta\leq t<0$). 
\item \underline{expanding}: if the mean curvature vector of $\phi(S\times\{t\})$ points \underline{away from} $S$, for all $0<t\leq\delta$ (resp.  for all $ -\delta\leq t<0$). 
\end{enumerate}
We say that the tubular neighborhood $\phi(S\times[-\delta,\delta])$ is \underline{contracting},  \underline{expanding} or \underline{mixed}, if the halves are contracting, expanding or of different types, respectively. 
\end{definition}

\begin{proposition}\label{index1}
Let $C$ be a compact cylindrical region contained in the interior of a $3$-manifold $(M,g)$ with negative sectional curvature, i.e. $C\simeq S\times [0,1]$, where $S$ is a closed surface of genus at least $2$, and $C\subset \operatorname{int} M$. Assume that
\begin{enumerate}
\item there are no stable minimal surfaces in the isotopy class of $S\times \{t\}$ contained in the interior of $C$, and
\item $S_{0}=S\times \{0\}$ and  $S_1=S\times \{1\}$ are stable minimal surfaces with contracting halved tubular neighborhoods contained in $C$.
\end{enumerate}
Then,  the interior of $C$ contains a minimal surface of Morse index 1 in the isotopy class of $S\times\{t\}$.
\end{proposition}

\begin{proof}
Without loss of generality, we can assume that, for some $\delta>0$ small, $S\times [0,\delta]$ and $S\times [1-\delta,1]$ correspond to the halved tubular neighborhoods of $S_0$ and $S_1$, that are contained in $C$ and are contracting with respect to $g$.

Let $(M,g_n)$ be a sequence of bumpy metrics converging to $(M,g)$ in the smooth topology and such that $S_0$ and $S_1$ are strictly stable minimal surfaces in $(M,g_n)$, for all $n\in \N$. The existence of this sequence follows from the combination of \cite[Proposition 2.3]{IrieMarquesNeves} (see formula for spectrum at the end of the proof of \cite[Proposition 2.3]{IrieMarquesNeves}) and the smooth bumpy metric theorem from \cite{WhiteBumpy2}. Additionally, choosing $n$ large enough, we can assume that $(M,g_n)$ has negative sectional curvature and the mean curvature vectors of $S_\delta=S\times \{\delta\}$ and $S_{1-\delta}=S\times \{1-\delta\}$, point towards $S_0$ and $S_1$ respectively. 

For the definition of \textit{sweepout} and \textit{saturared family} we refer the reader to \cite[Section 2]{MarquesNevesDuke}. We can form a canonical sweepout in the following way. First, by \Cref{theoremNegative}, applied to $S_\delta$ in the metric $(M,g_n)$, there is a monotone isotopy from $S_\delta$ to a stable minimal surface $\Gamma_0 \subset S\times [0,\delta)$. This isotopy foliates the region enclosed between $\Gamma_0$ and $S_\delta$ by surfaces that are $g_n$-mean-convex and have area strictly less than $\operatorname{Area}(S_\delta,g_n)$ (see \Cref{remarkArea}). Considering this isotopy backwards in time, we have a sweepout connecting $\Gamma_0$ with $S_\delta$. Next, we extend it to a sweepout connecting $\Gamma_0$ and $S_{1-\delta}$ by simply following the surfaces $S_t=S\times\{t\}$, for $\delta\leq t\leq 1-\delta$. Finally, we can use \Cref{theoremNegative} once more. This time we apply it to the surface $S_{1-\delta}$, to produce an isotopy that connects it with a stable minimal surface $\Gamma_1$. This produces a foliation of the region between $S_{1-\delta}$ and $\Gamma_1$, by surfaces that are $g_n$-mean-convex and have area strictly less than $\operatorname{Area}(S_{1-\delta},g_n)$. Putting all the pieces together we obtain a sweepout of the region between $\Gamma_0$ and $\Gamma_1$ by surfaces with area bounded from above by $\max_{t\in[\delta,1-\delta]} \operatorname{Area}(S_{t},g_n)$.

The sweepout we just constructed consists of surfaces isotopic to the fiber $S\times\{t\}$. Therefore, all the sweepouts in the saturated family generated by it also consists of one parameter families of surfaces isotopic to $S\times\{t\}$. Since $\Gamma_0$ and $\Gamma_1$ are strictly stable minimal surfaces in the metric $g_n$, we can apply the local min-max theorem \cite[Theorem 10]{KetoverLiokumovichSong} to this saturated family. In this way, we obtain a minimal surface $\Sigma_n$ of index at most 1 and such that $$\max\{\operatorname{Area}(\Gamma_0,g_n),\operatorname{Area}(\Gamma_1,g_n)\} < \operatorname{Area}(\Sigma_n,g_n)\leq  \max_{t\in[\delta,1-\delta]} \operatorname{Area}(S_{t},g_n).$$
The topology of $\Sigma_n$ is determined by the results from \cite{Ketover}. In particular, there is a sequence of surfaces obtained by performing neck-surgeries  and discarding components on leaves of the sweepouts in the family. Each component of the sequence converges to a component of $\Sigma_n$ as a one-sheeted graph if the component of $\Sigma_n$ is orientable, or as a two-sheeted graph otherwise. It is a classical result that there are no non-orientable closed surfaces embedded in $S\times[-1,1]$ (see \cite[Corollary 4.5]{BredonWood}). Moreover, since $S\times \{t\}$ is incompressible and $S\times[-1,1]$ is irreducible, neck surgeries always leave a component isotopic to $S\times \{t\}$ and possibly many $2$-spheres (\Cref{deftrivial}). Since $(M,g_n)$ has negative curvature there are no minimal spheres. In particular, there care only be one incompressible component isotopic to $S\times \{t\}$ converging to $\Sigma_n$. We conclude that $\Sigma_n$ is obtained with multiplicity one and it is isotopic to $S\times \{t\}$.

Moreover, $\Sigma_n\cap S\times [\delta,1-\delta]\neq\emptyset$. Otherwise, $\Sigma_n$ would be contained in one the regions enclosed by $S_i$ and $\Gamma_i$, for $i=0$ or $1$. Since the regions are foliated by $g_n$-mean-convex surfaces, this would imply that $\Sigma_n=\Gamma_i$,  for $i=0$ or $1$. Contradicting the first inequality.

Finally, by Ben Sharp's compactness theorem \cite[Theorem A.6.]{Sharp}, $\Sigma_n$ converges to a $g$-minimal surface $\Sigma$ of index at most 1. Since $\Sigma\cap S\times [\delta,1-\delta]\neq\emptyset$, it follows from the maximum principle that $\Sigma$ cannot intersect $\partial C$, which is minimal. Therefore $\Sigma\subset \operatorname{int} C$. Since there are none one-sided closed surfaces in $S\times [-1,1]$, $\Sigma$ has to be two-sided. If the convergence $\Sigma_n\to \Sigma$ happens with multiplicity greater than 1 then \cite[Theorem A.6.]{Sharp} would imply that $\Sigma$ is stable. This would contradict our assumptions. Therefore, $\Sigma_n\to \Sigma$ is a graphical convergence, so $\Sigma$ is in the isotopy class of $S\times\{t\}$. Finally, since there are not stable minimal surfaces contained in $\operatorname{int} C$, it follows that $\Sigma$ must have index 1. \end{proof}

\begin{proposition}\label{folcyl}
Let $C$ be a compact cylindrical region contained in the interior of a $3$-manifold $M$ with negative sectional curvature, i.e. $C\simeq S\times [0,1]$, where $S$ is a closed surface of genus at least $2$, and $C\subset \operatorname{int} M$. Assume that
\begin{enumerate}
\item there are only finitely many stable minimal surfaces in the isotopy class of $S\times \{t\}$ contained in the interior of $C$, and
\item $S_{0}=S\times \{0\}$ and  $S_1=S\times \{1\}$ are stable minimal surfaces with contracting halved tubular neighborhoods contained in $C$.
\end{enumerate}
Then, $C$ is foliated by surfaces in the isotopy class of $S\times\{t\}$, the foliation starts at $S_0$ and ends at $S_1$, and each leaf is either minimal or have non-vanishing mean curvature. 
\end{proposition}

\begin{proof}
Let $\{\Sigma_n\}_{n=1}^{N+1}$, with $S_0=\Sigma_1$ and $S_1=\Sigma_{N+1}$, be a maximal disjoint family of stable minimal surfaces contained in $C$ that belong to the isotopy class of $S\times\{0\}$. Without loss of generality we can assume that they are ``ordered'', i.e. there are cylindrical regions $\{C_i\}_{i=1}^N$ such that $\partial C_i=\Sigma_i\cup \Sigma_{i+1}$ and $C=\cup_{i=1}^N C_i$. Since the family is maximal the interior of the cylindrical regions $C_i$ is free of 
stable minimal surfaces in the isotopy class of $S\times\{0\}$. We now show that, for all $i=1,\dots,N$, there is a foliation of $C_i$ with the desired properties and extremes at $\Sigma_i$ and $\Sigma_{i+1}$. 

\underline{Case 1:} If one of the boundaries, let say $\Sigma_i$, has an expanding halved neighborhood in $C_i$ then applying \Cref{theoremNegative} to the boundary of the expanding neighborhood, produces a foliation ending at $\Sigma_{i+1}$. 

\underline{Case 2:} If both $\Sigma_i$ and $\Sigma_{i+1}$ have contracting neighborhoods then \Cref{index1} implies there exists a minimal surface $\Gamma \subset \operatorname{int} C_i$ of index 1 in the isotopy class of $S\times \{t\}$. Pushing normally by the first eigenfunction of the Jacobi operator of $\Sigma$ produces an expanding halved tubular neighborhood towards both sides of $\Sigma$. Then we apply the idea from Case 1 on each side.
\end{proof}

\subsection{Proof of \Cref{theoremFoliation}}

\begin{proof}
Assume $M\simeq S\times \R$ is quasi-Fuchsian. By the work of \cite{MazzeoPacard} each end of $M$ is foliated by constant mean curvature surfaces isotopic to $S\times\{t\}$, with mean curvature pointing towards the convex core of $M$.  Applying \Cref{theoremNegative} to one of this constant mean curvature leaves, produces a foliation ending at a stable minimal surface in the isotopy class of $S\times\{t\}$. If the halved tubular neighborhood outside of the foliation is expanding, we can repeat the argument using the boundary of such expanding end. By Theorem \Cref{theoremAnderson2} this process must end after finitely many times, until we reach stable surface with a tubular neighborhood that is contracting on each side. Doing the same from the other end, we end up with either two surfaces that coincide, or with a region satisfying the properties of \Cref{folcyl}.

Now assume that $M$ is fibered over $\mathrm{S}^1$. We can suppose that $M$ is not foliated by minimal fibers, otherwise there is nothing to prove. Because $M$ is irreducible, by \cite{MeeksSimonYau} there exists a surface $S\subset M$ isotopic to the fiber that is area minimizing in $M$. Cutting $M$ along $S$ we obtain a cylindrical region, which by \Cref{theoremAnderson1} satisfies the properties of \Cref{folcyl}. 
\end{proof}

\section{Outermost minimal surfaces in quasi-fuchsian manifolds}\label{secOutermost}

Let $S$ be a closed orientable surface of genus $g\geq 2$. There exists a set of hyperbolic metrics on $S\times \R$, called \textit{quasi-Fuchsian metrics}, that form the setting of a rich deformation theory. In \Cref{theoremFoliation} we showed that for such metrics, each end of $S\times \R$ admits a mean-convex foliation ending at an outermost minimal surface in the isotopy class of $S\times \{t\}$. 

Our goal in this section, is to study the continuity of this surface with respect of deformations of the metric. We begin by summarizing some properties of the space of quasi-Fuchsian metrics which are necessary for our results.  Detailed references are presented accordingly.

\subsection{Properties of the space of quasi-Fuchsian metrics} From now on, we denote by $\mathcal{M}$ the set of quasi-Fuchsian metrics on $S\times \R$, where $S$ is as in the beginning of the section. The following are well-known facts:

\begin{enumerate}
\item $\mathcal{M}$ admits a topology which makes it homeomorphic to $\R^{6g-6}\times \R^{6g-6}$. 
\item If $\{h_i\}_{i\in\N} \subset \mathcal{M}$ converges to $h \in \mathcal{M}$ on this topology, then we can choose a large smooth compact set $C\subset S\times \R$, such that: 
\begin{enumerate}
\item [(a)] $C$ is diffeomorphic to $S\times [-1,1]$.
\item [(b)] $(C,h_i)\to(C,h)$ in the smooth topology.
\item [(c)] $(\partial C,h)$ is strictly mean-convex.
\item [(d)] Every closed minimal surfaces of $(S\times \R,h_i)$
 and $(S\times \R,h)$ is contained in $(C,h_i)$ and $(C,h)$, respectively.
\end{enumerate}
\item For every $h \in \mathcal{M}$, $(S\times \R,h)$ contains at least one locally area-minimizing embedded surface in the isotopy class of $S\times\{t\}$. 
\item A metric $h$ in the diagonal of $\mathcal{M}\simeq \R^{6g-6}\times \R^{6g-6}$, is called \textit{Fuchsian}. If $h$ is Fuchsian, then $(S\times \R,h)$ contains exactly one closed minimal surface. 
\item If $h$ and $h'\in \mathcal{M}\simeq \R^{6g-6}\times \R^{6g-6}$ are images of each other under the symmetry along the diagonal of $\R^{6g-6}\times \R^{6g-6}$, then there exists an isometry $(S\times \R,h) \to (S\times \R,h')$ isotopic to the map $(p,t)\mapsto (p,-t)$.
\end{enumerate}

\begin{proof}
(1) is the well-known Bers' simultaneous uniformization theorem, that parametrizes the set of quasi-Fuchsian metrics by the conformal structures of the surfaces at the infinity of each end. The space of such conformal structures is the Teichmuller space which is homeomorphic to $\R^{6g-6}$ (see \cite{Uhlenbeck}). There are several notions of convergence for hyperbolic spaces all of which are equivalent in the case of 3-dimensional quasi-Fuchsian metrics. The type of limits described (2)(b) are called \textit{geometric limits} (see \cite[Section 2.2]{McMullenBook}). Each quasi-Fuchsian manifold has a compact convex core, which is a compact set containing all closed minimal surfaces (see \cite[Corollary 5.6]{Anderson}). The convex cores converge geometrically (see \cite[Corollary 7.34]{MatTan} and the discussion right before the corollary). Moreover, from \cite{MazzeoPacard}, we can pick any large smooth compact set cylinder $C$ containing the convex core and strictly mean-convex boundary. Together, these show (2)(a),(2)(c) and (2)(d). For discussions about (3) and (4) we refer the reader to \cite{Uhlenbeck}. Finally, (5) comes from the uniqueness and symmetry of Bers' simultaneous uniformization.
\end{proof}

The following proposition follows directly from the work of Taubes \cite{Taubes}. In fact, we only need the result for stable minimal surfaces in the fiber. That case is also easily derived from the earlier work of Uhlenbeck \cite{Uhlenbeck}.

\begin{proposition}
There exists an open set $\mathcal{M}^*\subset \mathcal{M}$, having total measure on $\mathcal{M} \simeq \R^{6g-6}\times \R^{6g-6}$, such that every $h\in \mathcal{M}^*$ is \emph{bumpy along the fiber}, i.e., $(S\times \R,h)$ contains no degenerate minimal surfaces in the isotopy class of $S\times \{t\}$.
\end{proposition}

There are quasi-Fuchsian metrics admitting many minimal surfaces:
\begin{theorem}\label{zeno}\cite[Theorem 1.5 and Section 5.2]{HuangWang4}
For every $N\in \N$, there exists $h\in \mathcal{M}$ such that $(S\times \R,h)$ contains at least $N$ disjoint locally area minimizing minimal surfaces in the isotopy class of $S\times \{t\}$.
\end{theorem}

Given $h\in \mathcal{M}$, we denote by $\Sigma_{out}(h)$ the outermost minimal surface obtained from applying \Cref{theoremNegative} to a mean-convex slice deep in the end $\big(S\times (0,+\infty),h\big)$, i.e. the positive end. Let $I(h)$ and $II(h)$ be the metric and second fundamental form of $\Sigma_{out}(h)$ with respect to $(S\times \R,h)$, respectively. We notice that in our context all the surfaces in the isotopy class of $S\times \{t\}$ have a defined \textit{marking} associated to the quasi-Fuchsian space $\mathcal{M}$. Using this marking we denote by $\mu(h)$ the unique hyperbolic metric in the conformal class of $I(h)$ and $u(h)$ the smooth function such that $I(h)=e^{2u(h)}\mu(h)$. Equivalently, we can think of $\mu(h)$ as a class of complex structures on $S$ modulo diffeomorphisms isotopic to the identity. Since $\Sigma_{out}(h)$ is minimal in $(S\times \R,h)$, following \cite[Section 4]{Uhlenbeck}, we can represent the second fundamental form $II(h)$ as a holomorphic quadratic differential $\alpha(h)$ on $(S,\mu(h))$. Moreover, it is shown in \cite[Section 4]{Uhlenbeck} that $u(h)=u(\mu(h),\alpha(h))$. 

\

In what follows, we study the continuity properties of the map $$h\mapsto (\mu(h),\alpha(h)).$$

\begin{remark} The space of parameters $(\mu,\alpha)$, consisting of $\mu$, a hyperbolic metric on $S$, and $\alpha$, a holomorphic quadratic differential, is identified with the cotangent bundle of Teichmuller space $T^*\mathcal{T}$. This is a $12g-12$ dimensional smooth real manifold. In \cite{Uhlenbeck}, Uhlenbeck constructs a bijective map from a neighborhood of the zero section $(\mu,0)$ to a set of quasi-Fuchsian metrics known as \textit{Almost-Fuchsian metrics}. The  map $$h\mapsto (\mu(h),\alpha(h)),$$ which we study, extends the inverse of Uhlenbeck's bijection.
\end{remark}

\begin{proposition}
The map $h\mapsto (\mu(h),\alpha(h))$ is injective on $\mathcal{M}$. \end{proposition}

\begin{proof}
First notice that $\Sigma_{out}(h)$ is not only incompressible, but has the same fundamental group of $S\times \R$. Then, the representation map constructed in \cite[Theorem 5.5]{Uhlenbeck} recovers the quasi-Fuchsian metric $h$ from $(\mu(h),\alpha(h))$. This shows the map $h\mapsto (\mu(h),\alpha(h))$ has an inverse from its image.
\end{proof}

\begin{proposition}\label{continuous}
If $\Sigma_{out}(h_0)$ is locally area-minimizing in $(S\times \R,h_0)$, then the map $h\mapsto (\mu(h),\alpha(h))$ is continuous at $h_0\in \mathcal{M}$. In particular, it is continuous on $\mathcal{M}^*$.
\end{proposition}

\begin{proof}
First, it is well known that if $\Sigma$ is a closed minimal surface of genus $g$ in a hyperbolic $3$-manifold, then
$$\mathrm{Area}(\Sigma) \leq 4\pi(g - 1).$$
Together, with the curvature bounds for stable minimal surfaces from \cite{Schoen}, this implies that a sequence of embedded stable minimal surfaces in the isotopy class of $S\times\{t\}$ converges to another stable minimal surface. We notice this is true even when the ambient metrics are also varying but converging smoothly to a limit metric. In addition, since there are no one-sided closed embedded surfaces in $S\times \R$, it follows that the convergence is graphical, therefore the limit belongs to the same isotopy class.

Let $\{h_i\}_{i\in\N}$ is a sequence of metrics converging to $h_0$. Since $\Sigma_{out}(h_0)$ is locally area-minimizing, and in view of theorem \Cref{theoremAnderson2}, it has a contracting tubular neighborhood in the sense of \Cref{defcontracting}, which, by abuse of notation, we denote as $\Sigma_{out}(h_0)\times (-\e,\e)$. Therefore, for $i$ large enough $(\Sigma_{out}(h_0)\times (-\e,\e),h_i)$ is a strictly mean-convex cylinder. In particular, it contains a locally area-minimizing surface $\Gamma_i$ in the isotopy class of $S\times\{t\}$. By our previous discussion, $\Gamma_i$ converges to $\Sigma_{out}(h_0)$, which is the only closed minimal surface contained in $(\Sigma_{out}(h_0)\times (-\e,\e),h_0)$. Since $\Sigma_{out}(h_i)$ is the outermost surface, this implies that it is deeper inside of the end, when compared to $\Gamma_i$. In particular, $\Sigma_{out}(h_i)$ must converge to a stable minimal surface which is deeper inside of then when compared to $\Sigma_{out}(h_0)$. Since $\Sigma_{out}(h_0)$ is the only such surface, it follows that $\Sigma_{out}(h_i)\to \Sigma(h_0)$ smoothly. This proves the proposition.
\end{proof}

\begin{proposition}
If $\Sigma_{out}(h_0)$ is degenerate stable with respect to $(S\times \R,h_0)$, then $h_0\in \mathcal{M}$ is in the closure of the singular set of the map $h\mapsto (\mu(h),\alpha(h))$. 
\end{proposition}

\begin{proof}
Assume that the proposition does not hold. Then, the restriction of  $h\mapsto (\mu(h),\alpha(h))$ to a ball around $h_0$ is an injective, continuous function from an open ball of $\R^{12g-12}$ to the $(12g-12)$-dimensional smooth manifold of parameters $(\mu,\alpha)$. By the theorem of invariance of domain, the map must be an homeomorphism, and therefore open. However, as described by Uhlenbeck in \cite[Diagram 4.3, Theorem 4.4]{Uhlenbeck}, there are no stable minimal surfaces in any hyperbolic 3-manifold with data $(\mu(h_0),(1+\e)\alpha(h_0))$, for any $\e>0$. This contradicts the openness of the map.
\end{proof}

The following is a standard continuity argument.

\begin{corollary}\label{cordegenerate}
There exists $h_0\in \mathcal{M}$ such that $\Sigma_{out}(h_0)$ is degenerate stable.
\end{corollary}

\begin{proof}
Given $h\in \mathcal{M}$, denote by $N(h)$ the maximum number of disjoint locally area-minimizing surfaces in the isotopy class of $S\times \{t\}$, contained in $(S\times \R,h)$. From \Cref{theoremAnderson2} and \Cref{theoremNegative} we know $$1\leq N(h)<+\infty.$$

Now consider a smooth curve $t\mapsto h_t \in \mathcal{M}$, for $0\leq t\leq 1$ starting at a Fuchsian metric $h_0$ and ending at a quasi-Fuchsian metric $h_1$ given by \Cref{zeno}. Therefore $N(h_0)=1$ and $N(h_1)>1$. 

Let $$\tau=\sup\{t\in [0,1]: N(h_s)=1, \forall s\leq t\}.$$ 

First, we show that $N(h_\tau)=1$. Otherwise $N(h_\tau)>1$, so there are disjoint locally area-minimizing surfaces in $(S\times\R,h_\tau)$. We can select disjoint tubular neighborhoods for the surfaces that are contracting. For metrics near $h_\tau$, these neighborhoods are still cylinders $S\times(-\e,\e)$ with mean convex boundary. Therefore, they each contain at least one locally area-minimizing surface. This contradicts the fact that $N(h_{\tau-\e})=1$, for $\e>0$ small enough.

Therefore, $\Sigma_{out}(h_\tau)$ is the only closed locally-area minimizing minimal surface of $(S\times\R,h_\tau)$. If $\Sigma_{out}(h_\tau)$ is degenerate then we are done. If not, then it would be strictly stable and by the Inverse function theorem there is a fixed tubular neighborhood  of $\Sigma_{out}(h_\tau)$ which is contracting for all metrics close to $h_\tau$. By definition there is a sequence $\e_i>0$, such that $(S\times \R,h_{\tau+\e_i})$ contains at least two stable minimal surfaces. Both surfaces cannot converge towards $\Sigma_{out}(h_\tau)$ as $n\to\infty$. Otherwise, they would eventually be contained in the contracting tubular neighborhood. Which is a contradiction. Since $N(h_\tau)=1$, one of the sequence converges to a degenerate stable which is not locally-areaminimizing.
\end{proof}

Combining \Cref{continuous} and \Cref{cordegenerate} we obtain

\begin{corollary}
The singular set of the map $h\mapsto (\mu(h),\alpha(h))$ is not empty. In particular, there are infinitely many $h_0\in \mathcal{M}$, such that $\Sigma_{out}(h_0)$ is degenerate stable with a tubular neighborhood of mixed type.
\end{corollary}

Finally, we study the continuity of the area of the outermost.

\begin{proposition}\label{area}
The map $h\mapsto \operatorname{Area}(\Sigma_{out}(h))$ is upper semicontinuous.
\end{proposition}

\begin{proof}
If $\Sigma_{out}$ is continuous at $h_0$ there is nothing to prove. Assume that $h_0$ is a point of discontinuity for the outermost and let $\{h_i\}_{i\in\N}$ be a sequence of quasi-Fuchsian metrics converging to $h_0$. Since $\Sigma_{out}(h_0)$ has a mixed tubular neighborhood, we can find another minimal surface below it using the level set flow (e.g. in \Cref{figQF} we would find III running the flow from IV). Notice that the area of this surface is strictly less than $\Sigma_{out}(h_0)$. Since there are only finitely many stable surfaces, we can repeat this process until we reach a locally area-minimizing surface $\Sigma$. For $i$ large enough, the metrics $h_i$ all contain a locally area minimizing surface near $\Sigma$. In particular, the outermost $\Sigma_{out}(h_i)$ has to be deeper in the end. Therefore, as $i\to\infty$, $\Sigma_{out}(h_i)$ must converge to a surface deeper in the end compared to $\Sigma$. By maximum principle it must be one of the surfaces we just constructed by running the flow from $\Sigma_{out}(h_0)$. It follows that $$\lim_{i\to\infty}\operatorname{Area}(\Sigma_{out}(h_i))\leq  \operatorname{Area}(\Sigma_{out}(h_0)).$$

\end{proof}

\bibliography{references}
\bibliographystyle{siam}

\end{document}